\documentclass[a4paper,11pt]{article}
\usepackage[T1]{fontenc}
\usepackage[english]{babel}
\usepackage{amsmath,amssymb,amsthm}
\usepackage{amscd}
\usepackage[shortlabels]{enumitem}
\RequirePackage[colorlinks,citecolor=blue,urlcolor=blue]{hyperref}
\newtheorem{theorem}{Theorem}[section]
\newtheorem{corollary}[theorem]{Corollary}
\newtheorem{lemma}[theorem]{Lemma}
\newtheorem{fact}[theorem]{Fact}
\newtheorem{proposition}[theorem]{Proposition}
\newtheorem{problem}[theorem]{Problem}

\theoremstyle{definition}
\newtheorem{definition}[theorem]{Definition}

\newtheorem{remark}[theorem]{Remark}
\newtheorem{example}[theorem]{Example}

\numberwithin{equation}{section}

\newcommand{\norm}[1]{\left\|{#1}\right\|}
\newcommand{\abs}[1]{\left|{#1}\right|}
\newcommand{\scalarp}[1]{\langle{#1}\rangle}
\newcommand{\scal}[2]{{\left\langle{{#1}\,\vert\,{#2}}\right\rangle}}

\DeclareMathOperator{\argmin}{argmin}

\DeclareMathOperator{\dom}{dom}
\newcommand{\minimize}[2]{\ensuremath{\underset{\substack{{#1}}}%
{\text{\rm minimize}}\;\;#2 }}

\newcommand{\genf}{h}
\newcommand{\hh}{\mathcal{H}}
\newcommand{\B}{\mathcal{B}}

\newcommand{\prox}{\mathrm{prox}}

\newcommand{\R}{\mathbb{R}}
\newcommand{\KK}{\mathbb{K}}
\newcommand{\RP}{\ensuremath{{\mathbb R}_+}}
\newcommand{\RX}{\ensuremath{\left]-\infty, +\infty\right]}}
\newcommand{\RPP}{\ensuremath{{\mathbb R}_{++}}}
\newcommand{\N}{\mathbb{N}}

\begin{document}

\title{ {\bf The variable metric forward-backward splitting 
algorithm under mild differentiability 
assumptions 
}}
\author{Saverio Salzo\\[1ex]
{\small\em LCSL, Istituto Italiano di Tecnologia and
Massachusetts Institute of Technology}\\[-0.5ex] 
{\small\em Via Morego 30, 16163 Genova, Italy}\\[-0.5ex]  
{\small\em email: saverio.salzo@iit.it}
}
\date{}
\maketitle
{\abstract 
We study the variable metric forward-backward splitting algorithm for convex minimization
problems without the standard assumption of the Lipschitz continuity of the gradient.
In this setting, we prove that, by requiring only 
mild assumptions on the smooth part of the objective function
and using several types of line search procedures 
for determining either the gradient descent stepsizes, or
the relaxation parameters, one still obtains
weak convergence of the iterates and convergence in the objective function values. 
Moreover, the $o(1/k)$  convergence rate in the function values is obtained
if slightly stronger differentiability assumptions are added. 
We also illustrate several applications
including 
problems that involve
Banach spaces and functions of divergence type. }

\vspace{1ex}

\section{Introduction}

The forward-backward splitting algorithm \cite{Smms05} is nowadays a
well-established and 
widely used first order optimization method
that is well suited for an objective function composed
by a smooth convex function plus a (possibly nonsmooth) convex simple
function. 
This algorithm
has been studied in a number of works \cite{Livre1,Tebu09,Bre09,Smms05,Dav2015}
which prove weak convergence of the iterates as well as $o(1/k)$ convergence rate
in function values.  The \emph{variable metric} version of the forward-backward method aims at accelerating the convergence of the standard algorithm.
It was first proposed in \cite{Che97} and its
global convergence property has been
established in full generality in \cite{Comb2014} where, 
under an appropriate monotonicity condition on the metrics, the authors
prove weak convergence of the iterates.
The same algorithm is also analyzed in \cite{Cho14,Fra15} for the minimization of \emph{non convex} composite functions and global convergence is achieved by assuming 
the Kurdyka-\L ojasiewicz property, together with compactness conditions,
and suitably controlling the behavior of the variable metrics. 

In this context, a fundamental and commonly adopted assumption 
 is that the gradient of the smooth part is Lipschitz continuous on the entire space.  
 However, there are a number of applications in which this 
condition is not satisfied: for instance, in inverse problems when the data fidelity term is
 based on Banach norms \cite{Bre09,Sch2012} or Bregman distances 
 (e.g., the Kullback-Leibler divergence, which is the appropriate 
 choice when the data are corrupted by Poisson noise \cite{Bon12,Sal14}).

\subsection{Objective and main contribution}
In this paper we address the convergence analysis of 
the variable metric forward-backward splitting algorithm in infinite dimensional Hilbert spaces 
without the assumption of 
the Lipschitz continuity of the gradient and using different
types of line search procedures.
This study provides global convergence guarantees, both in terms of convergence of the iterates
and rates of convergence in function values,
and shows that the scope of applicability of the algorithm
is significantly wider than that for which it was originally devised,
up to cover problems involving Banach spaces and objective functions of divergence type.
Our analysis is based on a general convergence principle for 
 abstract variable metric descent algorithms which blends the concept 
 of quasi-Fejer sequence with that of a sufficient decreasing condition. 
This principle simultaneously drives the convergence in the iterates and
the convergence in the objective function values. Moreover, we provide
a unifying view on several inexact line search 
procedures that have been proposed in literature in the context of projected/proximal
gradient-type algorithms clarifying the relationships among them.
We finally remark that,  even under standard differentiability assumptions,
we advance the related state of the art,
since we provide rate of convergence in function values in infinite dimensional setting
and we consider an alternative hypothesis on the metrics apart that of monotonicity
made in \cite{Bon15a,Comb2014}.


Our contribution is detailed below. We consider the problem
\begin{equation}
\tag{P}
\minimize{x \in \hh}{f(x) + g(x)},
\end{equation}
where, $\hh$ is a real Hilbert space and
\begin{description}
\item[H1] $f\colon \hh \to \RX$ and $g\colon \hh \to \RX$ 
are proper convex and lower semicontinuous functions 
with $\dom g \subset \dom f$;
\item[H2] $f$ is G\^ateaux differentiable on $\dom g$ and
$\nabla f$
is uniformly continuous on any weakly compact subset of $\dom g$.
\end{description}

For that problem, we study the following 
algorithm \cite{Comb2014}.
Let $x_0 \in \dom g$ and set
\begin{equation}
\tag{VM-FBS}
\begin{array}{l}
\text{for}\;k=0,1,\ldots\\[0.7ex]
\left\lfloor
\begin{array}{l}
\text{choose $\gamma_k \in \RPP$}\\[0.7ex]
y_k = \prox^k_{\gamma_k g}(x_k  - \gamma_k \nabla^k f(x_k))\\[0.7ex]
\text{choose $\lambda_k \in \left]0,1\right]$}\\[0.7ex]
x_{k+1} =  x_k + \lambda_k (y_k - x_k),\\[0.7ex]
\end{array}
\right.\\
\end{array}
\end{equation}
where $\nabla^k$  and $\prox^k$ denote the gradient operator and 
 the proximity operator with respect to a given family of scalar products
 as specified by the following assumption.
 \smallskip
 \begin{description}[itemsep=0mm]
\item[H3]
${(\scal{\cdot}{\cdot}_k)}_{k \in \N}$ is a sequence of possibly
varying scalar products (metrics) on $\hh$, with induced norms
 ${(\norm{\cdot}_k)}_{k \in \N}$  
 and associated positive operators ${(W_k)}_{k \in \N}$
 (i.e., for every $k \in \N$, $W_k\colon \hh \to \hh$ is such that
 $\scal{\cdot}{\cdot}_k = \scal{\cdot}{W_k\cdot}$),
 and
\begin{equation}
\label{eq:20170223a}
\exists\, (\nu,\mu) \in \R^2,\quad 0< \nu \leq \mu,\qquad (\forall\, k \in \N)\quad\ 
 \nu \norm{\cdot}^2 \leq \norm{\cdot}^2_k  \leq \mu\norm{\cdot}^2.
\end{equation}
\end{description}
 Moreover, since we are dropping out the assumption of Lipschitz continuity of $\nabla f$, 
 we rely on inexact line search methods for determining
 the parameters $\gamma_k$ and $\lambda_k$ in (VM-FBS).
We set, for every $k \in \N$ and every $x \in \dom g$, $\gamma>0$, and 
$\lambda \in \left]0,1\right]$
\begin{equation*}
J_k(x, \gamma, \lambda) 
= x + \lambda \big( \prox^k_{\gamma g}(x  - \gamma \nabla^k f(x)) - x \big),
\end{equation*}
so that $x_{k+1} = J_k(x_k, \gamma_k, \lambda_k) $.
Then, the gradient descent \emph{stepsizes} ${(\gamma_k)}_{k \in \N}$
and the \emph{relaxation} parameters ${(\lambda_k)}_{k \in \N}$ are 
chosen according
to one of the following rules: 
\smallskip
\begin{enumerate}[{\rm LS$1$}]
\item\label{LS1}
$0< \inf_{k \in \N} \lambda_k \leq \sup_{k \in \N} \lambda_k \leq 1$.
Let $\delta,\theta \in \left]0,1\right[$, $ \bar{\gamma}>0$,
and, $\forall\, k \in \N$,
\begin{equation*}
\begin{aligned}
\gamma_k &= \max\Big\{ \gamma \in \RPP \,\big\vert\,  
(\exists\, i \in \N)(\gamma = \bar{\gamma} \theta^i) \\
\qquad &  f(J_k(x_k,\gamma,\lambda_k)) - f(x_k) - \scal{J_k(x_k,\gamma,\lambda_k) -x_k}{\nabla f(x_k)} 
\leq \frac{\delta}{\gamma \lambda_k} \norm{J_k(x_k,\gamma,\lambda_k) - x_k}_k^2 \Big\}.
\end{aligned}
\end{equation*}
\item\label{LS2}
 $0< \inf_{k \in \N} \gamma_k \leq \sup_{k \in \N} \gamma_k <+\infty$.
Let $\delta,\theta \in \left]0,1\right[$, $\bar{\lambda} \in \left]0,1\right]$,
and,  $\forall\, k \in \N$,
\begin{equation*}
\begin{aligned}
\lambda_k &= \max\Big\{ \lambda \in \RPP \,\big\vert\,  
(\exists\, i \in \N)(\lambda = \bar{\lambda} \theta^i) \\
&\quad  f(J_k(x_k,\gamma_k,\lambda)) - f(x_k)
 - \scal{J_k(x_k,\gamma_k,\lambda) -x_k}{\nabla f(x_k)}
\leq \frac{\delta}{\gamma_k \lambda} \norm{J_k(x_k,\gamma_k,\lambda) - x_k}_k^2 \Big\}.
\end{aligned}
\end{equation*}
\item\label{LS4}
 $0< \inf_{k \in \N} \gamma_k \leq \sup_{k \in \N} \gamma_k <+\infty$.
Let $\theta, \delta \in \left]0,1\right[$, $\bar{\lambda} \in \left]0,1\right]$,
and, $\forall\, k \in \N$,
\begin{align*}
\lambda_k &= \max\Big\{ \lambda \in \RPP \,\big\vert\,  
(\exists\, i \in \N)(\lambda = \bar{\lambda} \theta^i)\\
& (f+g)(J_k(x_k,\gamma_k,\lambda)) - (f+g)(x_k) 
\leq (1-\delta)\lambda \big(g(y_k) - g(x_k) 
+ \scal{y_k -x_k}{\nabla f(x_k)} \big)\Big\}.
\end{align*}
\item\label{LS3}
$0< \inf_{k \in \N} \lambda_k \leq \sup_{k \in \N} \lambda_k \leq 1$.
Let $\delta,\theta \in \left]0,1\right[$, $\bar{\gamma}>0$,
and,  $\forall\, k \in \N$,
\begin{align*}
\gamma_k &= \max\Big\{ \gamma \in \RPP \,\big\vert\,  (\exists\, i \in \N)(\gamma 
= \bar{\gamma} \theta^i) \\
&\hspace{9ex}
\big\lVert \nabla^k f(J_k(x_k,\gamma,\lambda_k)) - \nabla^k f(x_k)\big\rVert_k 
\leq \frac{\delta}{\gamma \lambda_k} 
\norm{J_k(x_k,\gamma, \lambda_k) - x_k}_k\Big\}.
\end{align*}
\end{enumerate}

We remark that  \ref{LS1} and \ref{LS3}
search for an appropriate stepsize parameter $\gamma_k$ before setting $y_k$,
and choose a priori the relaxation parameters $\lambda_k$'s;  while
\ref{LS2} and \ref{LS4}, after computing $y_k$ with an a priori choice of $\gamma_k$, 
search for a suitable relaxation parameter $\lambda_k$. Note that \ref{LS1} and \ref{LS2}
ask for the descent lemma (Fact~\ref{fact_desclem}) for $f$ to be locally satisfied, and
 \ref{LS3} attempts to locally fulfill a Lipschitz condition for $\nabla f$.

Our analysis guarantees that, under the 
mild hypotheses {\bf H1}--{\bf H3} and additional suitable assumptions
on the metrics --- either {\bf H4} or {\bf H5} in Section~\ref{sec:convthms} --- (which are in line with those of \cite{Bon15a,Bon15,Comb2014}), each of the
line search rule above
makes algorithm (VM-FBS) capable of 
generating a minimizing sequence for $f+g$ 
 that also weakly converges to a solution of problem (P).
More precisely, denoting by $S_*$ the set of solutions of (P),
we prove that
\begin{itemize}
\item if $S_*= \varnothing$, then $(f+g)(x_k) \to \inf_{\hh}(f+g)$ and $\norm{x_k}\to +\infty$.
\item if $S_*\neq \varnothing$,  then
 $(x_{k})_{k \in \N}$ and $(y_{k})_{k \in \N}$ weakly converge to the same point
in $S^*$ and $(f+g)(y_k) \to \inf_\hh (f+g)$;
if, in addition, $\nabla f$ is Lipschitz continuous on the weakly compact subsets of $\dom g$,  
then $(\gamma_k)_{k \in \N}$ and 
$(\lambda_k)_{k \in \N}$ are
bounded away from zero and 
$\big( (f+g)(x_k) - \inf_{\hh}(f+g)\big) = o(1/k)$.
\end{itemize}
As a consequence, the above conclusions are also valid when
$\nabla f$ is $L$-Lipschitz continuous on $\dom g$
and the sequences ${(\gamma_k)}_{k \in \N}$ and  ${(\lambda_k)}_{k \in \N}$ are
chosen a priori (without backtracking) provided that 
they are bounded away from zero and  $\sup_{k \in \N} \gamma_k \lambda_k/\nu_k <2/L$,
 $\nu_k$ being the minimum eigenvalue of the metric $\scal{\cdot}{\cdot}_k$.
This result is not covered by the convergence analysis in 
\cite{Comb2014,Smms05,Comb15},
since they are based on the theory of fixed point algorithms for averaged nonexpansive
operators and the Baillon-Haddad theorem \cite{Bau10}, which requires 
$\nabla f$ to have full domain. Moreover, a coupling of 
the parameters $\gamma_k$, $\lambda_k$,
and $\nu_k$
is disclosed which somehow complements the analogue result in \cite{Comb15}. 

Finally we show that the above results hold also for  
 general placements of the domains of $f$ and $g$ 
 (thus even if $\dom g \not\subset \dom f$) at the cost of
 requiring coercivity-type conditions and
 adding a further line search procedure 
 that carry the iterates inside the domain of $f$ before executing any of \ref{LS1}--\ref{LS3}.
 This further generalization allows to treat, e.g., linear inverse problems with Poisson noise, that
 requires smooth terms of divergence type.

\subsection{Comparison with related works}

In  literature concerning the forward-backward algorithm,
the problem of removing the requirement of the Lipschitz continuous
 gradient for the smooth part 
has been receiving growing attention during the last years.
Currently there are two streams of research on this issue.
The first one focuses on a restricted class of smooth functions 
that do not enjoy Lipschitz continuous gradient, but possess other special properties. 
The work
\cite{Tra15} belongs to this kind of studies:
for $\hh$ finite dimensional, it analyses the case that 
$f$ is a smooth self-concordant function and addresses
both global and local convergence.
The second research line considers a wide class of smooth functions
(e.g., continuously differentiable) and introduce line searches 
 to determine the parameters of the algorithm.
Our work is within this stream. In the following we discuss two aspects.

\subsubsection*{The forward-backward algorithm without the Lipschitz assumption}

The literature on the forward-backward algorithm in the absence of
the Lipschitz assumption is scarse.
The pioneering work by Tseng and Yun \cite{Tseng09} is the 
first that considers a 
variable metric forward-backward algorithm 
in finite dimensional spaces, where the smooth part is only continuously differentiable (possibly non-convex). 
They proposes a general Armjio-type line search rule and prove that 
cluster points of the generated sequence are stationary points. 
Special instances of this general line search are also employed in the
recent works \cite{Cruz15,Bon15a} which advances the theory for the convex case
by addressing global convergence of the iterates and rate of convergence
in function values without the Lipschitz assumption. 
However, these studies
are not completely satisfying since in
\cite{Cruz15} the proposed line searches are not quite suitable for the forward-backward algorithm 
(as we discuss below), the $o(1/k)$ 
convergence rate in function values is obtained only in finite dimension,
and the differentiability assumptions are 
not completely relaxed,
so that, e.g., functions of divergent type may remain out of scope;
while  in 
\cite{Bon15a} the analysis is conducted in finite dimensional spaces and 
demands $\dom g$ to be closed and still the Lipschitz assumption for the $O(1/k)$ rate of convergence
in function values. 
 On the other hand, 
 the special case of the \emph{gradient projection method} \cite{Gol64,Lev66,Ius03} have been studied for long time
by requiring just the continuity of the gradient of the objective function and
 using different types of line searches for determining the step lengths  \cite{Ber76,Ber95,Cal87,Gaf82,McC72}. 
In particular, for the convex and finite dimensional case, \cite{Ius03}
proves convergence of the iterates 
using two types of Armijo
line searches, while in the recent \cite{Bon15},
both convergence of the iterates and a $O(1/k)$
convergence rate in function values are proved 
for the variable metric (scaled) version, assuming coercivity of the objective function and
local Lipschitz continuity of the gradient.

\subsubsection*{Line search methods for proximal gradient-type methods}

Line search methods for gradient-type algorithms date several decades 
ago \cite{Ber76,Cal87,Gaf82,Gaf84}.
In the following we position \ref{LS1}--\ref{LS3} 
with respect to the more recent literature.
We notice that, when  $g$ is zero, \ref{LS1}, \ref{LS2},
and \ref{LS4} reduce to the classical Armijo line search
along the steepest descent direction \cite{Ber95,Bur95,Noc06}.
Moreover, when
$g$ is the indicator of a closed convex set, \ref{LS4}
reduces to the  Armijo line search along the feasible direction 
commonly used in gradient projection 
methods \cite{Ber95,Bon15,Ius03}.
\ref{LS1} (with $\delta=1/2$ and $\lambda_k \equiv 1$) has been proposed
for the first time 
in \cite{Tebu09}, where the authors provide the rate of convergence in function values
of the forward-backward algorithm with no relaxation 
 under the assumption that $\nabla f$ is everywhere defined
and globally Lipschitz continuous. In that case the line search 
was
introduced to cope with situations in which the Lipschitz constant of $\nabla f$
was unknown or expensive to compute. 
\ref{LS4} (with $\gamma_k\equiv 1$) is a special instance of the general Armijio rule proposed
by Tseng and Yun \cite{Tseng09} --- which we do not explicitly treat here, but is encompassed by our analysis  (see Remark~\ref{eq:rmk20160409}).
\ref{LS4} has also
 been employed  in \cite{Lee14}
within a proximal Newton-type method for convex minimization problems 
of type (P) in finite dimensional spaces and  under Lipschitz assumption.
We recall that (VM-FBS) 
can be seen as a proximal quasi-Newton method: indeed the $y_k$ 
can be equivalently computed as
\begin{equation*}
y_k = \argmin_{z \in \hh} \scal{z - x_k}{\nabla f(x_k)} + \frac{1}{2 \gamma_k} \scal{z - x_k}{W_k (z-x_k))} + g(z).
\end{equation*}
Unfortunately the proof of global convergence in \cite[Theorem~3.1]{Lee14} is not correct. 
There, it is only proved
 the descent property and that $x_{k+1} - x_k \to 0$,
and mistakenly infered from this that $(x_k)_{k \in \N}$ is convergent.
Finally, \ref{LS3} was originally employed 
by Tseng in \cite{Tseng00} for the more general problem of finding zeros of the sum of 
two maximal monotone operators. This line search rule has been also recently studied in \cite{Cruz15} (with $\delta \leq 1/2$) in conjunction with
the (stationary metric) forward-backward splitting algorithm for convex minimization problems
without the assumption of Lipschitz continuous gradient.
However, we stress that procedure \ref{LS3}, at each iteration, 
calls for multiple evaluations  of the gradient 
--- a fact that may lead to significantly increase the computational cost
of the algorithm ---
 and, more importantly
 it may determine shorter stepsizes than those computed
 by the other line search rules
 (see Remarks~\ref{rmk:20160114b} and \ref{rmk:20160110a}).
In this regard, we note that
the procedure proposed in \cite{Tseng00}
is designed for general Lipschitz (or even continuous) monotone operators,
not specifically for gradient operators. By contrast, 
\ref{LS1}--\ref{LS4}
seem more appropriate to exploit the fact that we are dealing with
gradient operators:
we demonstrate indeed that they provide 
larger stepsizes that are consistent with those permitted under the standard Lipschitz
assumptions
(this issue parallels 
that between Lipschitz continuity and cocoercivity).
In \cite{Cruz15} a further line search is also analyzed
which is in between \ref{LS2} and \ref{LS4} with $\delta=1/2$, 
but again
 it leads to determine reduced step lengths
(see Remark~\ref{eq:rmk20160409}).

\subsection{Outline of the paper}
Section~\ref{sec:definition} contains notations and basic concepts and facts. 
In Section~\ref{sec:convanal} we first give preliminary results
concerning  the differentiability assumptions and the well-posedness of 
the line searches \ref{LS1}--\ref{LS3}; then we present an
abstract principle which is at the basis of 
the global convergence properties of variable metric descent algorithms;
and finally we study the convergence of 
algorithm (VM-FBS) in conjuction with the proposed line search
procedures.  Section~\ref{sec:gendomains} 
 shows that
the convergence results can be extended to situations in which the domain
of $g$ is not contained in the domain of $f$,
relaxing one requirement in {\bf H1}.
Finally, in Section~\ref{sec:app} we present examples of 
problems of type $(P)$, where the gradient of the smooth part is 
not Lipschitz continuous, that can be tackled by
the proposed algorithm.

\section{Basic definitions and facts}
\label{sec:definition}

Throughout the paper the notation we employ is standard 
and as in \cite{Livre1}.
We assume that $\hh$
is a real Hilbert space with scalar product $\scal{\cdot}{\cdot}$ 
and norm  $\norm{\cdot}$. 
If $\scal{\cdot}{\cdot}_1$ is another scalar product on $\hh$,
its \emph{associated positive operator}
is $V_1\colon \hh \to \hh$ 
such that, for every $(x,y) \in \hh^2$, $\scal{x}{y}_1 = \scal{x}{V_1 y}$.
We 
set $\RP = \left[0,+\infty\right[$ and 
$\RPP = \left]0,+\infty\right[$ and we denote by $\ell^1_+$
the set of summable sequences in $\RP$. Moreover, 
for every $(x,y) \in \hh^2$, we set
$[x,y] = \{x + t(y-x)  \,\vert\, t \in [0,1]\}$.
Let $h\colon \hh \to \RX$ be a proper function.
We set $\argmin_\hh h = \{x \in \hh~\vert~ h(x) = \inf_\hh h\}$,
and when it is a singleton, its unique element, by an abuse of notation,
is still denoted by $\argmin_\hh h$. 

We recall two fundamental facts about monotone sequences 
and Fej\'er sequences.

\smallskip
\begin{fact}[{\cite[Lemma~3]{Dav2015}}]
\label{l:20160110b}
Let ${(\rho_k)}_{k \in \N} \in \RP^{\N}$ and 
${(\alpha_k)}_{k \in \N} \in \RPP^{\N}$ be such that
\begin{equation*}
(\forall\, k \in \N)\ \  
\rho_{k+1} \leq \rho_{k} \quad\text{and}\quad   \sum\nolimits_{k \in \N} \alpha_k \rho_k <+\infty.
\end{equation*}
Then, for every $k \in \N$,
$\rho_k \leq \big(\sum_{i=0}^k \alpha_i\big)^{-1} \sum_{k \in \N} \alpha_k \rho_k$
and $\rho_k = o\big(1/\sum_{i=\lceil k/2 \rceil}^k \alpha_i \big)$.
In particular, if $(\alpha_k)_{k \in \N} \notin \ell^1_+$, then $\rho_k \to 0$, and if
$\inf_{k \in \N} \alpha_k>0$, then $\rho_k = o(1/k)$.
\end{fact}


\smallskip
\begin{definition}[{\cite[Definition~3.1(ii)]{Comb2013}}]
\label{def:fejer}
Let ${(\abs{\cdot}_k)}_{k \in \N}$ be a sequence of Hilbert norms on $\hh$
such that, for some  $\nu>0$, $\nu \lVert \cdot\rVert^2 \leq \lvert \cdot\rvert ^2_k$,
for every $k \in \N$.
Let $S \subset \hh$ be a nonempty set.
A sequence ${(x_k)}_{k \in \N}$ in $\hh$ is a 
\emph{quasi-Fej\'er  sequence with respect to $S$
relative to ${(\abs{\cdot}_k)}_{k \in \N}$}  if
there exist  ${(\varepsilon_k)}_{k \in \N}$
and ${(\eta_k)}_{k \in \N}$ in  $\ell_+^1$ such that
\begin{equation}
\label{eq:fejerprop}
(\forall\, x \in S)(\forall\, k \in \N)\qquad \abs{x_{k+1} - x}_{k+1}^2 
\leq (1+\eta_k) \abs{x_k - x}_k^2 + \varepsilon_k.
\end{equation}
\end{definition}
\begin{fact}[{\cite[Lemma~2.3, Proposition~3.2, and Theorem~3.3]{Comb2013}}]
\label{p:fejer}
Let ${(\abs{\cdot}_k)}_{k \in \N}$ be a sequence of Hilbert norms on $\hh$
with associated positive operators $(V_k)_{k \in \N}$.
Suppose that, for some $\nu>0$,
$\nu \norm{\cdot}^2 \leq \abs{\cdot}^2_k$ for every $k \in \N$.
Consider the following statements.
\begin{enumerate}[$(a)$]
\item\label{p:fejer_a} There exists  $(\eta_k)_{k \in \N} \in \ell_+^1$ 
such that, for every $k \in \N$, $\abs{\cdot}^2_{k+1} \leq (1+\eta_k) \abs{\cdot}^2_k$.
\item\label{p:fejer_b} There exists a positive operator $V$ such that,
for every $x \in \hh$, $V_k x \to V x$.
\end{enumerate}
Then \ref{p:fejer_a} $\Rightarrow$ \ref{p:fejer_b}.\!\!\!
\footnote{
The condition $\sup_{k \in \N} \norm{V_k}<+\infty$ required in 
\cite{Comb2013} is not necessary, since it is a consequence of 
$\abs{\cdot}^2_{k+1} \leq (1+\eta_k) \abs{\cdot}^2_k$
and $\prod_{k=0}^{+\infty}(1+\eta_k)<+\infty$.}
Moreover, let $S \subset \hh$ be a nonempty set, and let
$(x_k)_{k \in \N} \in \hh^{\N}$ be a 
quasi Fej\'er sequence with respect to $S$
relative to ${(\abs{\cdot}_k)}_{k \in \N}$.
Then,
 if \ref{p:fejer_b} holds,
\begin{enumerate}[{\rm(i)}]
\item\label{p:fejeri} 
$(x_k)_{k \in \N}$ is bounded and,
for every $y \in S$, ${(\abs{x_k - y}_k)}_{k \in \N}$ is convergent.
\item\label{p:fejerii} 
$(x_k)_{k \in \N}$ is weakly convergent to a point of $S$
if and only if every weak sequential cluster point of $(x_k)_{k \in \N}$ belongs to $S$.
\end{enumerate}
\end{fact}


\begin{definition}
Let $\genf\colon \hh \to \RX$
be  proper and convex and let $x \in \dom \genf$.
The function $\genf$ is said to be \emph{G\^ateaux differentiable} at $x \in \dom \genf$
if there exists $u \in \hh$, such that, for every
 $d \in \hh$, $\lim_{t\to 0} (\genf(x+td) - \genf(x))/t = \scal{d}{u}$
(see \cite[p.~243]{Livre1}); 
 in this case $\partial \genf(x) = \{u\}$ 
and the unique
 element of $\partial \genf(x)$ is denoted by $\nabla \genf(x)$.
Moreover $\genf$ is \emph{G\^ateaux differentiable on
$A \subset \hh$} if it is G\^ateaux differentiable at every point of $A$.
\end{definition}

\begin{remark}
\normalfont
\label{rmk:20160105a}
If  $h \in \Gamma_0(\hh)$ and it is G\^ateaux differentiable at $x \in \dom \genf$, then 
 $x \in  \mathrm{int} \dom \genf$ 
and $h$ is continuous on $\mathrm{int} \dom \genf \subset \dom \partial \genf$ 
\cite{Livre1}.
\end{remark}
\smallskip
\begin{fact}[Descent Lemma]
\label{fact_desclem}
Let $\genf\colon \hh \to \RX$
be  proper and convex. Suppose that
$\genf$ is G\^ateaux differentiable with $L$-Lipschitz continuous gradient 
on the segment $[x,y] \subset \mathrm{int} \dom \genf$.
Then, we have
$
\genf(y) - \genf(x) - \scal{y - x}{\nabla \genf(x)} 
\leq  (L/2) \norm{x - y}^2.
$
\end{fact}

Suppose that {\bf H3} holds. Let $k \in \N$ and let
$W_k$ be the positive operator associated to 
 $\scal{\cdot}{\cdot}_k$.
We set 
 $\mu_k = \sup_{\norm{x}= 1} \scal{x}{W_k x}=\norm{W_k} $ and 
$\nu_k = \inf_{\norm{x} = 1} \scal{x}{W_k x} = \lVert W_k^{-1}\rVert^{-1}$, being respectively
the maximum and minimum eigenvalue of $W_k$. Then
 \begin{equation}
\label{eq:norms}
\nu \norm{\cdot}^2\leq \nu_k \norm{\cdot}^2 \leq \norm{\cdot}^2_k 
\leq \mu_k\norm{\cdot}^2 \leq \mu\norm{\cdot}^2.
\end{equation}

Let $\genf \in \Gamma_0(\hh)$ and let $x \in \dom h$. 
We denote by $\partial^k h$ the subdifferential of $h$ at $x$ 
in the metric $\scal{\cdot}{\cdot}_k$
and we have, for every $u \in \hh$,
\begin{equation}
\label{eq:metric_subdiff}
u \in \partial^k \genf(x)\ \Leftrightarrow\ (\forall\, y \in \hh) 
\ \genf(y) \geq \genf(x) + \scal{y - x}{W_k u} 
\ \Leftrightarrow\ W_k u \in \partial \genf(x).
\end{equation}
Moreover, if $\genf$ is G\^ateaux differentiable at $x$ 
we denote by $\nabla^k h$ the gradient of $h$ at $x$ 
in the metric $\scal{\cdot}{\cdot}_k$, and we have
$\nabla \genf(x) = W_k \nabla^k \genf(x)$ and
\begin{equation}
\label{eq:20160204b}
(\forall\, d \in \hh)\quad 
\scal{d}{\nabla \genf (x)} = \lim_{t \to 0}\frac{\genf(x+t d) - \genf(x)}{t} 
= \big\langle d \,|\, \nabla^k \genf(x)\big\rangle_k.
\end{equation}
\smallskip
\begin{fact}
\label{p:20160125f}
Assume that {\bf H3} holds. Let $k \in \N$ and 
let $\nu_k$ and $\mu_k$ be the minimum and maximum eigenvalue
of $W_k$.
Let $\genf \in \Gamma_0(\hh)$ and suppose that
$\genf$ is G\^ateaux differentiable on a set $C \subset \hh$.
Let $x,y \in C$. Then
\begin{equation}
\label{eq:20160125e}
\frac{1}{\sqrt{\mu_k}} \norm{\nabla \genf(x) - \nabla \genf(y)}
\leq \big\lVert \nabla^k \genf(x) - \nabla^k \genf(y)\big\rVert_k
\leq \frac{1}{\sqrt{\nu_k}} \norm{\nabla \genf(x) - \nabla \genf(y)}.
\end{equation}
\end{fact}

Let $\genf \in \Gamma_0(\hh)$. The \emph{proximity operator of $\genf$} is the map
$\prox_\genf\colon\hh\to \hh$ such that, for every $x \in \hh$,
$\prox_\genf(x) = \argmin_{y \in \hh}
 \genf(y) + (1/2) \norm{y - x}^2$. Moreover
\begin{equation}
\label{eq:20160114a}
(\forall\, x \in \hh)(\forall\, z \in \hh)\qquad z = \prox_\genf x \iff x - z \in \partial \genf(z).
\end{equation}
The following result 
can be partially derived
 from the asymptotic behavior of the resolvent 
of maximal monotone operators 
\cite[Theorem~23.47]{Livre1}. 
We also provide the bound \eqref{eq:20160213e},
by slightly modifying the proof of \cite[Proposition~4.1.5, Chap.~XV]{Livre2}.

\begin{fact}
\label{lem:prox}
Let $\genf \in \Gamma_0(\hh)$ and let $\gamma >0$. 
Then, for every $u \in \dom h^*$, we have
\begin{equation}
\label{eq:20160213e}
(\forall\, x \in \dom \genf)\ \ 
\lVert \prox_{\gamma \genf}(x) - x\rVert^2 
\leq 2 \gamma \Big( \genf(x) + \genf^*(u) - \scal{x}{u} + \gamma \frac{\norm{u}^2}{2} \Big).
\end{equation}
In particular, for every $x \in \dom \genf$, $\prox_{\gamma \genf}(x) \to x$ as $\gamma \to 0^+$.
\end{fact}
\begin{proof}
Let $x \in \dom \genf$ and set, for the sake of brevity, $p_\gamma = \prox_{\gamma h}(x)$. 
It follows from \eqref{eq:20160114a} that 
$(x - p_\gamma)/\gamma \in \partial h(p_\gamma)$, hence
\begin{equation}
\label{eq:20151218a}
h(p_\gamma) + \frac{1}{\gamma} \norm{x - p_\gamma}^2
=  \Big\langle x - p_\gamma, \frac{x - p_\gamma}{\gamma} \Big\rangle + h(p_\gamma)
\leq h(x).
\end{equation}
Let $u \in \dom h^*$. Then, since
 $\scalarp{p_\gamma, u} - h^*(u) \leq h(p_\gamma)$, we have
\begin{align*}
&\frac{1}{2 \gamma} \norm{p_\gamma - x + \gamma u}^2 - \frac{\gamma}{2} \norm{u}^2
+\scalarp{x,u} - \genf^*(u) + \frac{1}{2 \gamma} \norm{p_\gamma - x}^2\\
&\quad =\scalarp{p_\gamma, u} - h^*(u) + \frac{1}{\gamma} \norm{p_\gamma - x}^2 \leq h(x).
\end{align*}
Hence 
\eqref{eq:20160213e} follows.
\end{proof}


\section{Convergence analysis}
\label{sec:convanal}

In this section 
we first 
discuss the hypotheses and the well-posedness of the 
 procedures \ref{LS1}--\ref{LS3}.
Then, we give a general convergence
principle for abstract variable metric descent algorithm (Theorem~\ref{thm:convergence02}).
Finally, we study the role and relationships among the proposed 
line search rules (Proposition~\ref{lem:20160129a}) and prove the convergence
properties of algorithm (VM-FBS) (Theorem~\ref{thm:convergence}).

\subsection{Preliminary results}

We examine assumption {\bf H2} and 
its consequences. 

\begin{remark}
\normalfont
\label{rmk:20151230a}
\begin{enumerate}[(i)]
\item\label{rmk:20151230ai}
If $\dom g = \hh$, {\bf H2}
is equivalent to requiring that $f$ is Fr\'echet differentiable on $\hh$ 
and that $\nabla f$ is uniformly continuous
on bounded sets (see Corollary~\ref{l:20160111w}\ref{l:20160111wii}). 
\item {\bf H2} is satisfied if $\nabla f$
is Lipschitz continuous on the weak compacts of $\dom g$.
\item By Remark~\ref{rmk:20160105a},
 {\bf H2} implies  
 $\dom g \subset \mathrm{int} \dom f$ and 
  $f$ is continuous on $\dom g$. 
\item\label{rmk:20151230aiii} {\bf H2} implies that the function 
$(\nabla f)_{\lvert \dom g} \colon \dom g \to \hh$
is continuous ---  in the relative topology of $\dom g$ 
(see Corollary~\ref{l:20160111w}\ref{l:20160111wi} below).
\item\label{rmk:20151230aibis}
Since
continuity on compact sets yields uniform continuity 
(Heine-Cantor theorem), 
if $\hh$ is finite dimensional,  hypothesis {\bf H2} turns to
require that $f$ is G\^ateaux differentiable on $\dom g$
and that ${(\nabla f)}_{\lvert \dom g} \colon \dom g \to \hh$
is continuous (with respect to the relative topology of $\dom g$).
\end{enumerate}
\end{remark}

The following lemmas are at the basis of our convergence analysis. 

\begin{lemma}
\label{l:20160103b}
Let $(x_k)_{k \in \N}$ and $(\tilde{x}_k)_{k \in \N}$ be two sequences in $\hh$,
let $\bar{x}\in \hh$ and suppose that $x_k \rightharpoonup \bar{x}$ 
and $\tilde{x}_k \rightharpoonup \bar{x}$.
Then $\{\bar{x}\} \cup \bigcup_{k \in \N} [x_k, \tilde{x}_k]$ is weakly compact.
In particular $\bigcup_{k \in \N} [x_k, \bar{x}]$ is weakly compact.
\end{lemma}
\begin{proof}
Let us denote by $\hh_w$ the space $\hh$ endowed with the
weak topology. We recall that $\hh_w$ is a locally convex space \cite{Ree80}.
Set $A= \{\bar{x}\} \cup \bigcup_{k \in \N} [x_k, \tilde{x}_k]$ and
let $(U_i)_{i \in I}$ be an open covering of $A$ in $\hh_{w}$. Then there exists
$i_* \in I$ such that $\bar{x} \in U_{i_*}$. Thus, 
since $U_{i_*}$ is a weak neighborhood of $\bar{x}$,
there exists a convex neighborhood $V$ 
of the origin in $\hh_w$ such that 
$\bar{x} + V + V \subset U_{i_*}$.
Since $\bar{x} + V$ is a weak neighborhood of
$\bar{x}$, $x_k \rightharpoonup \bar{x}$ and $\tilde{x}_k - x_k \rightharpoonup 0$, 
there exists $\nu \in \N$ such that for every 
integer $k >\nu$,
we have $x_k \in \bar{x} + V$ and $\tilde{x}_k - x_k \in V$;
hence, for every $t \in [0,1]$, 
$x_k + t(\tilde{x}_k - x_k) \in \bar{x} + V + V \subset U_{i_*}$.
Moreover, for every integer $k \leq \nu$, since $[x_k, \tilde{x}_k]$ 
is weakly compact, 
there exists a finite $I_k \subset I$ such that
$[x_{k},\tilde{x}_k] \subset \bigcup_{i \in I_k} U_{i}$. Eventually,
setting $\tilde{I} = \bigcup_{k=0}^\nu I_k$ (which is finite),
we have $A \subset U_{i_*} \cup \bigcup_{i\in \tilde{I}} U_{i}$. The second
part of the statement follows from the first part by just
taking $\tilde{x}_k = \bar{x}$, for every $k \in \N$.
\end{proof}

\begin{lemma}
\label{l:20160103a}
Let $\Omega$ 
be an open subset of $\hh$ and let $f\colon \Omega \to \R$.
Suppose that 
$f$ is G\^ateaux differentiable on a convex set $C\subset \Omega$.
Let $(x_k)_{k \in \N}$ and $(\tilde{x}_k)_{k \in \N}$ be
sequences in $C$ and let $\bar{x} \in C$  be
such that $x_k \rightharpoonup \bar{x}$
and $\tilde{x}_k - x_k \to 0$.
Then the following hold:
\begin{enumerate}[{\rm(i)}]
\item
\label{l:20160103ai} 
Suppose that $\nabla f$ is uniformly continuous on any weak compact of 
$C$.
Then 
\begin{enumerate}[$(a)$]
\item\label{l:20160103aia}
 $\norm{\nabla f(\tilde{x}_k) - \nabla f(x_k)} \to 0$;
\item\label{l:20160103aib}
 $(\forall\, \varepsilon>0)(\exists\, \delta>0)(\forall\, k \in \N)$\\
$\norm{\tilde{x}_k - x_k} \leq \delta\ \Rightarrow\ 
\lvert f(\tilde{x}_k) - f(x_k) - \scal{\tilde{x}_k - x_k}{\nabla f(x_k)} \rvert 
\leq \varepsilon \norm{\tilde{x}_k - x_k}$.
\end{enumerate}
\item
\label{l:20160103aii} 
Suppose that $\nabla f$ is Lipschitz continuous on any weakly compact subset of $C$.
Then there exists $L>0$ such that, for every $k \in \N$,
$\norm{\nabla f(\tilde{x}_k) - \nabla f(x_k)} 
\leq L \norm{\tilde{x}_k - x_k}$
and 
$\lvert f(\tilde{x}_k) - f(x_k) - \scal{\tilde{x}_k - x_k}{\nabla f(x_k)} \rvert 
\leq L \norm{\tilde{x}_k - x_k}^2/2$.
\end{enumerate}
\end{lemma}
\begin{proof}
It follows from Lemma~\ref{l:20160103b} that 
$A =\{\bar{x}\} \cup \bigcup_{k \in \N} [x_k, \tilde{x}_k]$ is weakly compact,
and moreover $A\subset C$.

\ref{l:20160103ai}:
That $\norm{\nabla f(\tilde{x}_k) - \nabla f(x_k)} \to 0$ follows from the
fact that $\nabla f$ is uniformly continuous on $A$, 
that, for every $k \in \N$, $\tilde{x}_k, x_k \in A$,
and that $\tilde{x}_k - x_k \to 0$.
Let $\varepsilon>0$. Then,
since $\nabla f$ is uniformly continuous on $A$,
there exists $\delta>0$ such that, for every $x,y \in A$,
$ \norm{x - y} \leq \delta\ \Rightarrow\ \norm{\nabla f(y) - \nabla f(x)} \leq \varepsilon$.
Let $k \in \N$ be such that $\norm{\tilde{x}_k - x_k} \leq \delta$. Then
 $[x_k,\tilde{x}_k] \subset A$ and
 the function $t \mapsto f(x_k + t(\tilde{x}_k - x_k))$
is differentiable with derivative $t \mapsto \scal{\tilde{x}_k-x_k}{\nabla f(x_k + t(\tilde{x}_k-x_k))}$,
which is continuous (for $\nabla f$ is uniformly continuous on $A$).
 Therefore, we have
\begin{align*}
\lvert f(\tilde{x}_k) &- f(x_k) - \scal{\tilde{x}_k - x_k}{\nabla f(x_k)} \rvert \\
&= \Big\lvert \int_0^1 
\scal{\tilde{x}_k - x_k}{\nabla f(x_k + t(\tilde{x}_k-x_k)) - \nabla f(x_k)} d t \Big\rvert
 \leq \varepsilon \norm{\tilde{x}_k - x_k}.
\end{align*}

\ref{l:20160103aii}: Since $\nabla f$
is Lipschitz continuous on $A$, there exists $L>0$ such that,
for every $x, y \in A$, $\norm{\nabla f(x) - \nabla f(y)} \leq L \norm{x - y}$.
Let  $k \in \N$. Then
$\norm{\nabla f(\tilde{x}_k) - \nabla f(x_k)} 
\leq L \norm{\tilde{x}_k - x_k}$. Moreover,
since $[x_k,\tilde{x}_k] \subset A$, arguing as before, we have
\begin{equation*}
\lvert f(\tilde{x}_k) - f(x_k) - \scal{\tilde{x}_k - x_k}{\nabla f(x_k)} \rvert 
 \leq \int_0^1 L t \norm{\tilde{x}_k - x_k}^2 d t = \frac{L}{2} \norm{\tilde{x}_k - x_k}^2.
\end{equation*}
\end{proof}

\begin{corollary}
\label{l:20160111w}
Let $\Omega$ 
be a nonempty open subset of $\hh$ and let $f\colon \Omega \to \R$.
Suppose that 
$f$ is G\^ateaux differentiable on a nonempty convex set $C \subset \Omega$ and that
$\nabla f$ is uniformly continuous on any weakly compact subset of $C$.
Then
\begin{enumerate}[{\rm(i)}]
\item\label{l:20160111wi}  
${(\nabla f)}_{\lvert C}\colon C \to \hh$ is continuous (in the relative topology of $C$).
\item\label{l:20160111wii} 
for every $\bar{x} \in C$
\begin{equation*}
\lim_{\substack{x \to \bar{x}\\x \in C, x\neq \bar{x}}} 
\frac{\lvert f(x) - f(\bar{x}) - \scal{x-\bar{x}}{\nabla f(\bar{x})}\rvert}{\norm{x-\bar{x}}} =0.
\end{equation*}
\end{enumerate}
\end{corollary}
\begin{proof}
\ref{l:20160111wi}:
For every $\bar{x} \in C$ and every $(\tilde{x}_k)_{k \in \N}$ in $C$ such that 
$\tilde{x}_k \to \bar{x}$, it follows from 
Lemma~\ref{l:20160103a}\ref{l:20160103ai}\ref{l:20160103aia} 
(with $x_k = \bar{x}$, for every $k \in \N$) that 
$\norm{\nabla f(\tilde{x}_k) - \nabla f(\bar{x})} \to 0$.

\ref{l:20160111wii}:
Let $\bar{x} \in C$. Then for every $(\tilde{x}_k)_{k \in \N}$ in 
$C\setminus\!\{\bar{x}\}$ such that 
$\tilde{x}_k \to \bar{x}$, 
Lemma~\ref{l:20160103a}\ref{l:20160103ai}\ref{l:20160103aib} 
(with $x_k \equiv \bar{x}$) yields
$\lvert f(\tilde{x}_k) - f(\bar{x}) - \scal{\tilde{x}_k-\bar{x}}{\nabla f(\bar{x})}\rvert/
\norm{\tilde{x}_k-\bar{x}} \to 0$.
\qquad\end{proof}

\begin{lemma}
\label{lem:Jproperties}
Assume that {\bf H1} holds
and that $f$ is G\^ateaux differentiable on $\dom g$.
 Let $x \in \dom g$ and set, for every $\gamma \in \RPP$ and $\lambda \in \left]0,1\right]$,
 \begin{equation}
\label{eq:Jdef}
J (x,\gamma,\lambda) = x 
 + \lambda \big( \prox_{\gamma g} (x - \gamma \nabla f(x)) -x \big).
\end{equation}
Then the following hold.
\begin{enumerate}[{\rm(i)}]
\item 
\label{lem:Jpropertiesi}
Let $\lambda\in \left]0,1\right]$. Then for every $(\gamma_1,\gamma_2) \in \RPP^2$,
\begin{equation*}
\gamma_1 \leq \gamma_2\ \Rightarrow\ \norm{J(x,\gamma_1,\lambda) - x} \leq \norm{J(x,\gamma_2, \lambda) - x} 
\leq \frac{\gamma_2}{\gamma_1} \norm{J(x,\gamma_1, \lambda) - x}.
\end{equation*}
\item 
\label{lem:Jpropertiesi2}
Let $\gamma\in \RPP$. Then for every $(\lambda_1,\lambda_2) \in \left]0,1\right]^2$, 
\begin{equation*}
\lambda_1 \leq \lambda_2\ \Rightarrow\ \norm{J(x,\gamma,\lambda_1) - x} \leq \norm{J(x,\gamma, \lambda_2) - x} 
= \frac{\lambda_2}{\lambda_1} \norm{J(x,\gamma, \lambda_1) - x}.
\end{equation*}
\item 
\label{lem:Jpropertiesii}
$\forall\, \lambda \in \left]0,1\right]$ 
$\lim_{\gamma\to 0+} J(x,\gamma,\lambda) = x$ 
and $\forall\,\gamma \in \RPP$
$\lim_{\lambda\to 0^+}J(x,\gamma,\lambda) = x$.
\item\label{lem:Jpropertiesiib}
Let $u \in \dom g^*$. Then
for every $x \in \dom g$ and every $\gamma>0$,
\begin{equation}
\label{eq:20160213c}
\norm{J(x,\gamma,1) - x} \leq \gamma\! \norm{\nabla f(x)} 
+  \sqrt{2 \gamma \big(g(x) \!+\! g^*(u) \!-\! \scal{x}{u} \!+\! \gamma \norm{u}^2/2\big)}.
\end{equation}
\end{enumerate}
\end{lemma}
\begin{proof}
\ref{lem:Jpropertiesi}: It is a consequence of the fact that 
$\gamma\mapsto \norm{J(x,\gamma,1) - x}$ is increasing and
that $\gamma\mapsto \norm{J(x,\gamma,1) - x}/\gamma$ is decreasing:
see \cite[Lemma~2.2]{Cal87} for the projection case and \cite[Lemma~3]{Tseng09} 
and \cite[Lemma~2.4]{Cruz15} for the general case. 

\ref{lem:Jpropertiesi2}:
It is trivial.

\ref{lem:Jpropertiesii}: 
We prove the first part. Let $\gamma \in \RPP$. Then, since 
 $\prox_{\gamma g}$ is non-expansive
\begin{equation}
\label{eq:20160213a}
\begin{aligned}
\frac{\norm{J(x,\gamma,\lambda) - x}}{\lambda} &\leq \norm{\prox_{\gamma g}(x - \gamma \nabla f(x)) 
- \prox_{\gamma g}(x)} + \norm{\prox_{\gamma g} (x)- x}\\
& \leq \gamma \norm{\nabla f(x)} + \norm{\prox_{\gamma g} (x)- x}.
\end{aligned}
\end{equation}
Therefore, we derive from Fact~\ref{lem:prox} that $\norm{J(x,\gamma,\lambda) - x} \to 0$
as $\gamma \to 0^+$. 

\ref{lem:Jpropertiesiib}:
It follows from \eqref{eq:20160213a} and Fact~\ref{lem:prox}. 
\qquad\end{proof}

Finally the following lemma addresses the well-posedness of 
the definitions of the various proposed line search procedures.

\begin{lemma}
\label{lem:linesearch}
Assume that {\bf H1} holds
and that $f$ is G\^ateaux differentiable on $\dom g$.
 Let $J$ as in \eqref{eq:Jdef}
and let $x \in \dom g$. Suppose that $x \notin \argmin (f+g)$. Then 
\begin{enumerate}[{\rm(i)}]
\item
\label{lem:Jpropertiesiii}
If $(\nabla f)_{\lvert \dom g}\colon \dom g \to \hh$ 
is continuous at $x$, then
\begin{equation}
\label{eq:20151220aii}
(\forall\, \lambda>0)\ \lim_{\gamma \to 0^+} 
\frac{\gamma  \norm{ \nabla f(J(x,\gamma,\lambda)) - \nabla f(x)}}
{\norm{J(x,\gamma) - x}} = 0.
\end{equation}
\item
\label{lem:Jpropertiesiv} 
If $\nabla f$ is uniformly continuous on any weakly compact 
subsets of $\dom g$, then
\begin{equation}
\label{eq:20151220a}
(\forall\, \lambda>0)\ \lim_{\gamma \to 0^+} \frac{\gamma  
\lvert f(J(x,\gamma, \lambda)) - f(x) - \scal{J(x,\gamma,\lambda) -x}{\nabla f(x)}\rvert}
{\norm{J(x,\gamma, \lambda) - x}^2} = 0
\end{equation}
and
\begin{equation}
\label{eq:20160125m}
(\forall\, \gamma>0)\ \lim_{\lambda \to 0^+} \frac{\lambda  
\lvert f(J(x,\gamma, \lambda)) - f(x) - \scal{J(x,\gamma, \lambda) -x}{\nabla f(x)}\rvert}
{\norm{J(x,\gamma, \lambda) - x}^2} = 0.
\end{equation}
\item 
\label{lem:Jpropertiesv}
Let  $\gamma>0$
and $y \in \dom g$. Suppose that $y - x$ is a descent direction
for $f+g$ at $x$, that is $(f+g)^\prime(x,y-x)<0$. Then,
for every $\delta \in \left]0,1\right[$ there exists $\lambda_0 \in \left]0,1\right]$
such that, for every $\lambda \in \left]0,\lambda_0\right]$,
\begin{equation*}
(f+g)(x + \lambda(y - x)) - (f+g)(x)
\leq (1 - \delta) \lambda
\big(g (y) - g(x) + \scal{y - x}{\nabla f(x)} \big).
\end{equation*}
\end{enumerate}
\end{lemma}
\begin{proof}
\ref{lem:Jpropertiesiii}-\ref{lem:Jpropertiesiv}:
We first note that, using \eqref{eq:20160114a}, for every $\gamma>0$, we have
\begin{equation*}
x \in \argmin (f+g) \iff - \nabla f(x) \in \partial g(x) \iff x = \prox_{\gamma g}(x - \gamma \nabla f(x)).
\end{equation*}
Let $\lambda>0$. Then, since $x \notin \argmin (f+g)$, we have
 $\norm{J(x,\gamma,\lambda) - x} \neq 0$,
for every $\gamma>0$.
Next, since 
$\gamma \mapsto \norm{J(x,\gamma,\lambda) - x}/\gamma$ is decreasing,
$
\lim_{\gamma \to 0^+} \norm{J(x,\gamma,\lambda) - x}/\gamma
= \sup_{\gamma>0} \norm{J(x,\gamma,\lambda) - x}/\gamma>0,
$
and hence there exists $(\gamma_0,M) \in \RPP^2$ such that
\begin{equation}
\label{eq:20160102b}
\frac{\gamma}{\norm{J(x,\gamma,\lambda) - x}} \leq M \quad \forall\, \gamma \in \left]0, \gamma_0\right].
\end{equation}
Moreover $\lim_{\gamma \to 0^+}J(x,\gamma,\lambda) = x$.
Thus, if $(\nabla f)_{\lvert \dom g}$ is continuous at $x$ 
(case \ref{lem:Jpropertiesiii}), 
 \begin{equation}
\label{eq:20160102aii}
\lim_{\gamma \to 0^+} \norm{ \nabla f(J(x,\gamma,\lambda)) - \nabla f(x)} = 0,
\end{equation}
otherwise, if 
$\nabla f$ is uniformly continuous on any compact subsets of $\dom g$
(case \ref{lem:Jpropertiesiv}),
 we derive from Corollary~\ref{l:20160111w}\ref{l:20160111wii} that 
\begin{equation}
\label{eq:20160102a}
\lim_{\gamma \to 0^+} \frac{\lvert f(J(x,\gamma,\lambda)) - f(x) 
- \scal{J(x,\gamma,\lambda) -x}{\nabla f(x)}\rvert}
{\norm{J(x,\gamma,\lambda) - x}} = 0.
\end{equation}
Then, \eqref{eq:20151220aii} follows from \eqref{eq:20160102aii} and \eqref{eq:20160102b},
whereas \eqref{eq:20151220a} follows from \eqref{eq:20160102a} and \eqref{eq:20160102b}.\\
Since $\lim_{\lambda\to0^+}J(x,\gamma,\lambda) = x$ and
$(J(x,\gamma,\lambda) - x )/\lambda = 
\lVert \prox_{\gamma g} (x - \gamma \nabla f(x)) - x \rVert \neq 0$,
equation \eqref{eq:20160125m} follows from Corollary~\ref{l:20160111w}\ref{l:20160111wii},
as done before.

\ref{lem:Jpropertiesv}:
We have 
\begin{equation}
\label{eq:20160201d}
\frac{(f+g)(x + \lambda(y - x)) - (f+g)(x)}{\lambda} \to 
(f+g)^\prime(x, y - x)
\ \text{as}\ \lambda \to 0
\end{equation}
and, since for every $z \in \dom g$, 
$g^\prime(x,z - x) \leq g(z) - g(x)$ \cite[Proposition~17.2]{Livre1}, 
\begin{equation}
\label{eq:20160201b}
(f+g)^\prime(x, y - x) = g^\prime(x,y-x) + \scal{y-x}{\nabla f(x)}
\leq g(y) - g(x) + \scal{y-x}{\nabla f(x)}.
\end{equation}
Thus, if $-\infty < (f+g)^\prime(x, y - x)<0$, we have
$(f+g)^\prime(x, y - x) < (1-\delta) (f+g)^\prime(x, y - x)$ and hence,
by \eqref{eq:20160201b},
\begin{equation}
\label{eq:20160201c}
(f+g)^\prime(x, y - x) < 
(1-\delta) \big( g(y) - g(x) + \scal{y-x}{\nabla f(x)}\big).
\end{equation}
Otherwise, if $(f+g)^\prime(x, y - x)=-\infty$, clearly
\eqref{eq:20160201c} still holds.
Therefore, in any case, 
it follows from \eqref{eq:20160201d} and \eqref{eq:20160201c} that
there exists $\lambda_0>0$ such that, 
$\forall\,\lambda \in \left]0,\lambda_0\right]$, 
\begin{equation*}
\frac{(f+g)(x + \lambda(y - x)) - (f+g)(x)}{\lambda} 
\leq (1 - \delta)\big( g(y)-g(x) + \scal{y - x}{\nabla f(x)} \big).
\end{equation*}
\end{proof}


\begin{remark}
\normalfont
\begin{enumerate}[(i)]
\item 
In view of Lemma~\ref{lem:linesearch}\ref{lem:Jpropertiesiii}-\ref{lem:Jpropertiesv}
(applied with respect to each metric $\scal{\cdot}{\cdot}_k$) 
and  \eqref{eq:20160204b}, 
the line searches \ref{LS1}, \ref{LS2}, and \ref{LS3} 
are well-defined for every $\delta, \theta \in \left]0,1\right[$ and every
$\bar{\lambda} \in \left]0,1\right]$ and $\bar{\gamma}>0$; and, since $y_k - x_k$
is a descent direction for $f+g$, when $y_k \neq x_k$ (see the subsequent Lemma~\ref{lem:FBS_step}),  
\ref{LS4} is well-defined too.
\item 
The line search methods we presented have different computational costs.
\ref{LS2} and \ref{LS4} are the cheapest one
since they require just 
one evaluation of $\nabla f$
and  $\prox_g$;  \ref{LS1}  
requires multiple evaluation of the proximity operator of $g$, therefore it is feasible
when computing $\prox_g$ is cheap.
\ref{LS3} is the most costly since it demands also to compute 
$\nabla f$ multiple times.
\end{enumerate}
\end{remark}

\subsection{An abstract convergence principle}

We present an abstract convergence theorem underlying 
the different versions of the variable metric forward-backward
splitting algorithm we will consider. It uses the property \ref{prop:genconv_b}
below that blends the 
concept of quasi-Fej\'er sequence with that of a sufficient decreasing condition.
This result has the same flavor of that
 given in \cite[Section~2.3]{Att13}.

\begin{proposition}
\label{prop:genconv}
Let ${(\abs{\cdot}_k)}_{k \in \N}$ be a sequence of Hilbert norms on $\hh$
such that the sequence of the associated positive operators is strongly 
(that is pointwise) convergent in $(\hh, \norm{\cdot})$.
Suppose that there exists $\nu>0$ such that,
for every $k \in \N$, $\nu \norm{\cdot}^2 \leq \abs{\cdot}_k^2$.
Let $h \in \Gamma_0(\hh)$ and let $(x_k)_{k \in \N}$ be a sequence in $\dom h$.
Set $S_* = \argmin_\hh h$ and $S =\{x \in \hh \,\vert\, h(x) \leq \inf_{k \in \N} h(x_k)\}$.
Consider the following properties
\begin{enumerate}[$(a)$]
\item
\label{prop:genconv_a} 
$(h(x_k))_{k \in \N}$ is decreasing;
\item
\label{prop:genconv_b} 
There exist $(\alpha_k)_{k \in \N}$ in $\RPP^{\N}$, with 
$\sup_{k \in \N} \alpha_k<+\infty$, and 
$(\eta_k)_{k \in \N}$ and ${(\varepsilon_k)}_{k \in \N}$ in $\ell_+^1$,
such that
\begin{equation}
\label{eq:genconv}
(\forall\, x \in \dom h)(\forall\, k \in \N)\  \abs{x_{k+1} - x}_{k+1}^2 
\leq (1+\eta_k) \abs{x_k - x}_k^2 + 
2 \alpha_k \big( h(x) - h(x_{k+1})\big) + \varepsilon_k.
\end{equation}
\item
\label{prop:genconv_c} 
There exist
 $(y_k)_{k \in \N}$
and $(v_k)_{k\in \N}$ in $\hh^\N$ such that,  $\forall\, k \in \N$, $v_k \in \partial h(y_k)$
and for every weakly convergent subsequence $(x_{n_k})_{k \in \N}$
of $(x_k)_{k \in \N}$, $x_{k_n} - y_{k_n}\rightharpoonup 0$ and $v_{n_k}\to 0$.
\end{enumerate}
Then the following hold. 
\begin{enumerate}
\item Suppose that \ref{prop:genconv_a} is satisfied.
\begin{enumerate}[{\rm(i)}]
\item\label{prop:genconv_i} 
If $(x_k)_{k \in \N}$ admits a bounded subsequence,
then $\inf_{k \in \N} h(x_k)>-\infty$ and $S \neq \varnothing$.\footnote{
$S_* \subset S$ and, since $h$ is a proper function, $S\neq \varnothing 
\Rightarrow \inf_{k \in \N} h(x_k)>-\infty$.}
\item\label{prop:genconv_ibis}
If $\inf_{k \in \N} h(x_k)=-\infty$, then $\norm{x_k}\to +\infty$ and  $h(x_k) \to \inf_\hh h$.
\end{enumerate}
\item 
Suppose that \ref{prop:genconv_a} and \ref{prop:genconv_b} are satisfied.
\begin{enumerate}[resume*]
\item\label{prop:genconv_ii} 
If $\inf_{k \in \N} h(x_k)>-\infty$, then
$\sum_{k \in \N} \norm{x_{k+1} - x_k}^2<+\infty$.
\item\label{prop:genconv_iii} 
If $S\neq \varnothing$, then
$(x_k)_{k \in \N}$ is a quasi-Fej\'er sequence with respect to $S$
relative to the sequence of norms ${(\abs{\cdot}_k)}_{k \in \N}$.
\item\label{prop:genconv_vi}
Suppose that $S_*\neq \varnothing$. 
If $\sum_{k \in \N} \alpha_k = +\infty$,
then $h(x_{k}) \to \inf_\hh h$.
If  $\inf_{k \in \N} \alpha_k >0$, then $(h(x_{k}) - \inf_\hh h) = o(1/k)$.
\end{enumerate}
\item 
Suppose that \ref{prop:genconv_a}, \ref{prop:genconv_b}, 
and \ref{prop:genconv_c} are satisfied.
\begin{enumerate}[resume*]
\item\label{prop:genconv_iv} 
If $\bar{x}\in \hh$ and $(x_{n_k})_{k \in \N}$ is a subsequence of $(x_k)_{k \in \N}$ such that
$x_{n_k} \rightharpoonup \bar{x}$,
then $\bar{x} \in S_*$ 
and $h(y_{n_k})\to \inf_\hh h$.
\item\label{prop:genconv_v}
Suppose that $S_* \neq \varnothing$. Then 
$(x_k)_{k \in \N}$ and $(y_k)_{k \in \N}$ converge weakly to the same point in $S_*$
and $h(y_k)\to \inf_\hh h$.
\item\label{prop:genconv_vii}
Suppose that $S_*=\varnothing$. Then $\norm{x_k} \to +\infty$
and $h(x_k) \to \inf_{\hh}h$.
\end{enumerate}
\end{enumerate}
\end{proposition}
\begin{proof}
\ref{prop:genconv_i}:
Suppose that $(x_k)_{k \in \N}$ has a bounded subsequence.
Then there exists a subsequence 
$(x_{n_k})_{k \in \N}$ of $(x_k)_{k \in \N}$ and $\bar{x} \in \hh$ such that
$x_{n_k} \rightharpoonup \bar{x}$. Since $h$ is lower semicontinuous and 
$(h(x_k))_{k \in \N}$ is decreasing, 
 $-\infty < h(\bar{x}) \leq \liminf_{k} h(x_k) = \inf_{k\in \N} h(x_k)$.

\ref{prop:genconv_ibis}:
Since $\inf_{k \in \N} h(x_k)=-\infty$ and $(h(x_k))_{k \in \N}$ is decreasing,
then $h(x_k) \to -\infty = \inf_\hh h$. 
Moreover, it follows from \ref{prop:genconv_i} that
$(x_{k})_{k \in \N}$ does not have any bounded subsequence,
hence $\liminf_k \norm{x_k}\to +\infty$.

\ref{prop:genconv_ii}:
Taking $x = x_k$ in \eqref{eq:genconv}, we have
\begin{equation}
\label{eq:20160109a}
(\forall\, k \in \N)\qquad\nu \norm{x_{k+1}- x_k}^2 \leq \abs{x_{k+1}- x_k}_{k+1}^2 \leq 2 \alpha_k \big( h(x_k) - h(x_{k+1})\big) + \varepsilon_k.
\end{equation}
Now note that, if $\inf_{k \in \N} h(x_k)>-\infty$, the sequence 
${\big(h(x_k) - h(x_{k+1}) \big)}_{k\in \N}$ is summable (and positive).
Therefore, since $(\alpha_k)_{k \in \N}$ is bounded
and ${(\varepsilon_k)}_{k \in \N}$ is summable,
the right hand side of \eqref{eq:20160109a}
is summable.

\ref{prop:genconv_iii}:
Suppose that $S\neq \varnothing$ and let $x_* \in S$. Then it follows from \eqref{eq:genconv} with $x = x^*$ that,
for every $k \in \N$, 
$\abs{x_{k+1} - x_*}_{k+1}^2 \leq (1+\eta_k) \abs{x_k - x_*}_k^2 + \varepsilon_k$
(see Definition~\ref{def:fejer}).

\ref{prop:genconv_vi}:
Let $x_* \in S_*$. 
It follows from \ref{prop:genconv_iii} that
$(x_k)_{k \in \N}$
is a quasi-Fejer sequence with respect to $S_*$ relative to 
 ${(\abs{\cdot}_k)}_{k \in \N}$,  hence Fact~\ref{p:fejer}\ref{p:fejeri}
yields that ${\big( \abs{x_k - x_*}_k^2\big)}_{k \in \N}$ is bounded.
It follows from \eqref{eq:genconv} with $x = x_*$
that, for every $k \in \N$
\begin{equation}
\label{eq:20160125a}
2\alpha_k\big( h(x_{k+1}) - h(x_*) \big)
\leq \abs{x_k - x_*}_k^2 - \abs{x_{k+1} - x_*}_{k+1}^2 
+ \eta_k \abs{x_k - x_*}_k^2
+ \varepsilon_k.
\end{equation}
Thus, $(h(x_{k+1})- \inf_\hh h)_{k \in \N}$ is a 
sequence in $\R_+$ which is decreasing and moreover, since 
 ${\big( \abs{x_k - x_*}_k^2\big)}_{k \in \N}$ is bounded, the right hand side 
 of \eqref{eq:20160125a}
is summable, and hence the sequence
${\big(\alpha_k(h(x_{k+1})- \inf_\hh h)\big)}_{k \in \N}$ is summable. 
The statement follows from 
Fact~\ref{l:20160110b}.

\ref{prop:genconv_iv}: 
Let $\bar{x} \in \hh$ and let $(x_{n_k})_{k \in \N}$ be a subsequence
of $(x_k)_{k \in \N}$ such that $x_{n_k} \rightharpoonup \bar{x}$.
Since $y_{n_k} - x_{n_k} \rightharpoonup 0$, we have
$y_{n_k} \rightharpoonup \bar{x}$.
Then, since, for every $k \in \N$, $v_{n_k} \in \partial h(y_{n_k})$, $v_{n_k} \to 0$,
 and $\partial h$ is demiclosed \cite{Livre1},
then we have $0 \in \partial h(\bar{x})$, hence $\bar{x} \in S_*$. 
Moreover,  
$v_{n_k} \in \partial h(y_{n_k})$ yields, 
\begin{equation*}
(\forall\, k \in \N)\quad h(y_{n_k}) \leq h(\bar{x}) + \scal{ y_{n_k} - \bar{x}}{v_{n_k}}
\end{equation*}
and $\scal{ y_{n_k} - \bar{x}}{v_{n_k}} \to 0$, for 
$y_{n_k} \rightharpoonup \bar{x}$ and $v_{n_k} \to 0$.
Thus, by the lower semicontinuity of $h$,
$h(\bar{x}) \leq \liminf_k h(y_{n_k}) 
\leq \limsup_k h(y_{n_k}) \leq h(\bar{x})$,
that is $h(y_{n_k}) \to h(\bar{x}) = \inf_\hh h$.

\ref{prop:genconv_v}:
Suppose $S_*\neq \varnothing$. Then, by
\ref{prop:genconv_iii},  $(x_k)_{k \in \N}$ is a quasi-Fej\'er sequence
with respect to $S_*$ relative to ${(\abs{\cdot}_k)}_{k \in \N}$.
Moreover, \ref{prop:genconv_iv} yields that
every weak sequential cluster point of $(x_k)_{k \in \N}$
belongs to  $S_*$. Thus, it follows
from Fact~\ref{p:fejer}\ref{p:fejerii} that $(x_{k})_{k \in \N}$ converges
weakly to a point in $S_*$. Then, by applying 
 \ref{prop:genconv_iv} and  property \ref{prop:genconv_c} to 
the entire sequence $(x_{k})_{k \in \N}$
we derive that $h(y_{k}) \to \inf_\hh h$ and $y_k - x_k \rightharpoonup 0$.

\ref{prop:genconv_vii}: It follows from \ref{prop:genconv_iv}
that $(x_k)_{k \in \N}$ does not have any bounded subsequence.
Therefore $\liminf_k \norm{x_k} = +\infty$. If it was $\inf_k h(x_k)>\inf_\hh h$,
then the set $S$ would be nonempty and \ref{prop:genconv_iii}
would yield that $(x_k)_{k \in \N}$ is a quasi-Fej\'er sequence 
and hence bounded. Thus, necessarily $h(x_k) \to \inf_k h(x_k)=\inf_\hh h$.
\end{proof}

\begin{remark}
\normalfont
If in \eqref{eq:genconv} we consider stationary metrics and 
replace $h(x_{k+1})$ by $h(x_k)$,
we obtain the notion of
 \emph{modified Fej\'er sequences} 
 introduced in \cite{Lin15}.
The authors show that that concept
is useful to analyze the convergence in function values of 
 splitting algorithms. 
 However the convergence
of the iterates is not studied.
\end{remark}

Now we  give the general theorem of convergence for 
variable metric algorithms.

\begin{theorem}
\label{thm:convergence02}
Under the assumption of Proposition~\ref{prop:genconv},
 suppose that  \ref{prop:genconv_a}
in Proposition~\ref{prop:genconv} is satisfied and that, 
if $\inf_{k \in \N} h(x_k)>-\infty$, conditions
\ref{prop:genconv_b} and  \ref{prop:genconv_c}
in Proposition~\ref{prop:genconv} are satisfied for some $(\alpha_k)_{k \in \N}$
and $(y_k)_{k \in \N}$. 
Then the following hold.
\begin{enumerate}[{\rm(i)}]
\item\label{thm:convergence02ii}  If $\inf_{k \in \N} h(x_k)>-\infty$, then
$\sum_{k \in \N} \norm{x_{k+1} - x_k}^2<+\infty$.
\item\label{thm:convergence2i} 
Suppose that $S_* \neq \varnothing$. Then 
\begin{enumerate}[$(a)$]
\item  $(x_{k})_{k \in \N}$ and  $(y_{k})_{k \in \N}$ weakly converge 
to the same point in $S_*$.
\item $h(y_k) \to \inf_{\hh}h$.
\item If $\sum_{k \in \N} \alpha_k = +\infty$,
 then $h(x_k) \to \inf_{\hh}h$.
\item  If  $\inf_{k \in \N} \alpha_k> 0$, then
$ h(x_k) - \inf_{\hh}h = o(1/k)$.
\end{enumerate}
\item\label{thm:convergence2ii}  
If $S_* = \varnothing$, then 
 $\norm{x_k}\to +\infty$ and $h(x_k) \to \inf_\hh h$.
\end{enumerate}
\end{theorem}
\begin{proof}
If $\inf_{k \in \N} h(x_k) = -\infty$, we are in the case $S_* = \varnothing$
and the statement follows from Proposition~\ref{prop:genconv}\ref{prop:genconv_ibis}.
If $\inf_{k \in \N} h(x_k) > -\infty$, then conditions \ref{prop:genconv_a},
\ref{prop:genconv_b} and \ref{prop:genconv_c} in Proposition~\ref{prop:genconv}
are satisfied and the conclusions follow from Proposition~\ref{prop:genconv}.
\end{proof}

\subsection{Convergence theorems}
\label{sec:convthms}

In this section we finally address the convergence of (VM-FBS) with  
line searches \ref{LS1}--\ref{LS3}. In addition to {\bf H1}--{\bf H3}, we will also consider one of the following assumptions on the metrics.
\begin{description}[itemsep=0mm]
\item [ H4] There exists $(\eta_k)_{k \in \N} \in \ell_+^1$ 
such that, for every $k \in \N$, $\norm{\cdot}^2_{k+1} \leq (1+\eta_k) \norm{\cdot}^2_k$.
\item [ H5] $\sum_{k \in \N} (\mu_k - \nu_k)<+\infty$, where, $\forall\, k \in \N$, 
$\nu_k$ and $\mu_k$
are respectively the minimum and maximum eigenvalue of the positive operator $W_k$
associated to $\scal{\cdot}{\cdot}_k$.
\end{description}

\begin{remark}\
\normalfont
\begin{enumerate}[(i)]
\item
 {\bf H4} can be equivalently written as $W_{k+1} \preccurlyeq (1+\eta_k) W_k$ 
 and it was considered in \cite{Par08} for the proximal point algorithm
and in \cite{Bon15a,Comb2014} for the forward-backward algorithm. 
In view of Fact~\ref{p:fejer},  {\bf H4} implies that $W_k$ strongly converges
to some positive operator.
\item
{\bf H5} encompasses and generalizes the condition assumed in \cite{Bon15}
for the scaled gradient projection method, where the $W_k$'s
are indeed forced to converge to the identity operator
at certain rate. By contrast, we stress that {\bf H5} does not implies that
the $W_k$'s strongly converge: just take 
 $W_k = \mu_k \mathrm{Id}$ with $(\mu_k)_{k \in \N}$
 a non convergent bounded sequence in 
$\left[\varepsilon, +\infty \right[$ for some $\varepsilon>0$.
However, {\bf H5} implies that, as $k \to +\infty$, $W_k$ takes the
form of a multiple of the identity operator, but the multiplicative constant may continue
to vary with $k$.
\end{enumerate}
\end{remark}

The following result is fundamental 
and analyzes just one step of {\rm (VM-FBS)}.
So, we can avoid to refer to the variable metric  and state the
result
in the metric $\scal{\cdot}{\cdot}$ of the space $\hh$.
Items \ref{lem:FBS_stepi}--\ref{lem:FBS_stepii} below are standard and appear 
explicitly in \cite[Lemma~1]{Tseng09}.

\begin{lemma}
\label{lem:FBS_step}
Assume  {\bf H1} 
and that $f$ is G\^ateaux differentiable on $\dom g$.
Let $k \in \N$ and let $\gamma_k \in \RPP$
and $\lambda_k \in \left]0,1\right]$.
Let $x_k \in \dom g$ and set
$y_k = \prox_{\gamma_k g} (x_k - \gamma_k \nabla f(x_k))$
and $x_{k+1} =J(x_k, \gamma_k,\lambda_k) =  x_k +  \lambda_k (y_k - x_k)$.
Then the following hold.

\begin{enumerate}[{\rm(i)}]
\item \label{lem:FBS_stepi}
$\displaystyle g(y_k) - g(x_k) + \scal{y_k - x_k}{\nabla f(x_k)} \leq 
- \frac{\norm{y_k - x_k}^2}{\gamma_k}$.
\item \label{lem:FBS_stepii}
$\displaystyle(f+g)^\prime(x_k, y_k-x_k) \leq - \frac{\norm{y_k - x_k}^2}{\gamma_k}$; 
in particular if $y_k \neq x_k$, then $y_k - x_k$ is a descent direction for $f+g$.\\[-1ex]
\item\label{lem:FBS_stepiii}
$\forall\, x \in \dom g$,
\begin{multline*}
\norm{x_{k+1} - x}^2 \leq \norm{x_k - x}^2
+ 2 \gamma_k \lambda_k \big( (f+g)(x) - (f+g)(x_k)\big)\\
- 2 \gamma_k \lambda_k \big(g(y_{k})- g(x_k)  
+ \scal{y_{k} - x_k}{\nabla f(x_k)} \big) - \norm{x_{k+1} - x_k}^2.
\end{multline*}
\end{enumerate}
\end{lemma}
\begin{proof}
\ref{lem:FBS_stepi}:
By the definition of $y_k$ and \eqref{eq:20160114a} we derive that
$(x_k - y_k)/\gamma_k - \nabla f(x_k) \in \partial g(y_k)$, and hence
that, for every $x \in \dom g$,
\begin{align}
\nonumber
g(x)  
\nonumber&\geq g(y_k)  +  \Big\langle x - x_k~\Big|~\frac{x_k - y_k}{\gamma_k} 
- \nabla f(x_k) \Big\rangle
+ \Big\langle x_k - y_k~\Big|~\frac{x_k - y_k}{\gamma_k} - \nabla f(x_k) \Big\rangle\\
\label{eq:20160402b} & =  g(y_k)  +  \frac{1}{\gamma_k}\scal{x - x_k}{x_k - y_k} 
- \scal{x - y_k}{\nabla f(x_k)} +  \frac{\norm{x_k - y_k}^2}{\gamma_k}.
\end{align}
Taking $x = x_k$ in the above inequality, \ref{lem:FBS_stepi} follows.

\ref{lem:FBS_stepii}:
It follows from \ref{lem:FBS_stepi} and the fact that 
 $g^\prime(x_k,y_k - x_k) \leq g(y_k) - g(x_k)$.

\ref{lem:FBS_stepiii}:
Let $x \in \dom g$. Since $f$ is convex,
$f(x) - f(x_k) \geq  \scal{x - x_k}{\nabla f(x_k)}$.
Thus, it follows from \eqref{eq:20160402b}  
that
\begin{align*}
 (f+g)(x) - (f+g)(x_k) \geq g(y_{k}) - g(x_k) 
&+ \scal{y_{k} - x_k}{\nabla f(x_k)}\\
&+\frac{1}{\gamma_k}\scal{x - x_k}{x_k - y_{k}}
+  \frac{\norm{y_{k} - x_k}^2}{\gamma_k} .
\end{align*}
Now, multiplying the above inequality by $2 \gamma_k\lambda_k$
we obtain
\begin{equation*}
\begin{aligned}
2 \gamma_k \lambda_k \big( (f+g)(x) - (f+g)(x_k)\big) &\geq
2 \gamma_k \lambda_k \big(g(y_{k})- g(x_k) + \scal{y_{k}-x_k}{\nabla f(x_k)} \big)\\
&\qquad + 2 \scal{x - x_k}{x_k - x_{k+1}}
+2 \frac{\norm{x_k - x_{k+1}}^2}{\lambda_k} .
\end{aligned}
\end{equation*}
Finally, since 
$2 \scal{x - x_k}{x_k - x_{k+1}} = \norm{x_{k+1} - x}^2 
- \norm{x_k - x}^2 - \norm{x_{k+1} - x_k}^2$, 
\begin{align}
\nonumber 2 \gamma_k \lambda_k \big( (f+g)(x) &- (f+g)(x_k)\big) \geq
2 \gamma_k \lambda_k \big(g(y_{k})- g(x_k) + \scal{y_{k} - x_k}{\nabla f(x_k)}\big)\\ 
\label{eq:20160202e} &\qquad+ \norm{x_{k+1} - x}^2 
- \norm{x_k - x}^2 - \norm{x_{k+1} - x_k}^2
+2\frac{\norm{x_{k+1} - x_k}^2}{\lambda_k}
\end{align}
and hence, since $2/\lambda_k \geq 2$, the statement follows.
\qquad\end{proof}

In view of  Proposition~\ref{prop:genconv}\ref{prop:genconv_b} and 
Lemma~\ref{lem:FBS_step}\ref{lem:FBS_stepiii}, it is clear that
to obtain convergence of algorithm (VM-FBS), we need
to ensure that the positive quantity
\begin{equation}
\label{eq:20160203g}
- \big(g(y_{k})- g(x_k) + \scal{y_{k} - x_k}{\nabla f(x_k)} \big)
\end{equation}
is summable. 
Now we show how the various
line search methods are related to each other
and the role that they play in making 
\eqref{eq:20160203g} summable.

\begin{proposition}
\label{lem:20160129a}
Under the same hypotheses of Lemma~\ref{lem:FBS_step}, 
let $\delta \in \left]0,1\right[$
and consider the following statements
\begin{enumerate}[$(a)$]
\item\label{lem:20160129a_a} 
$\displaystyle
\norm{ \nabla f(x_{k+1}) - \nabla f(x_k)} 
 \leq \frac{\delta}{ \gamma_k \lambda_k} \norm{x_{k+1} - x_k}$.
\item\label{lem:20160129a_b}
$\displaystyle f(x_{k+1}) - f(x_k) - \scal{x_{k+1} -x_k}{\nabla f(x_k)}
 \leq \frac{\delta}{ \gamma_k \lambda_k} \norm{x_{k+1} - x_k}^2$.\\[-0.5ex]
 \item\label{lem:20160129a_c}
 $\displaystyle (f+g)(x_{k+1}) - (f+g)(x_k) \leq (1-\delta)\lambda_k \big(g(y_{k}) - g(x_k) 
 + \scal{y_{k} -x_k}{\nabla f(x_k)} \big)$.
\end{enumerate}

Then the following hold.
\begin{enumerate}[{\rm(i)}]
\item\label{lem:20160129ai} 
\ref{lem:20160129a_a} $\Rightarrow$ \ref{lem:20160129a_b}
 $\Rightarrow$ \ref{lem:20160129a_c}.
\item \label{lem:20160129aiii}
If \ref{lem:20160129a_c} holds, then,
for every $x \in \dom g$,
\begin{equation}
\begin{aligned}
\label{eq:20160202d}
\norm{x_{k+1} - x}^2 \leq \norm{x_k - x}^2
&+ 2 \gamma_k \lambda_k \big( (f+g)(x) - (f+g)(x_{k})\big)\\[0.5ex]
&+ \frac{2 \gamma_k}{1 - \delta}\big( (f+g)(x_{k}) - (f+g)(x_{k+1}) \big)
- \norm{x_{k+1} - x_k}^2.
\end{aligned}
\end{equation}
\item \label{lem:20160129aiv}
If  \ref{lem:20160129a_c} holds, then
\begin{equation*}
(1-\delta) \norm{x_{k+1} - x_k}^2 \leq \gamma_k \big( (f+g)(x_{k}) - (f+g)(x_{k+1}) \big);
\end{equation*}
in particular $(f+g)(x_{k+1}) \leq (f+g)(x_{k})$.
\end{enumerate}
\end{proposition}
\begin{proof}
\ref{lem:20160129ai}:
Suppose that \ref{lem:20160129a_a} holds. Then, using the convexity of $f$, we derive that
\begin{equation*}
\begin{aligned}
f(x_k) &
\geq f(x_{k+1}) + \scal{x_k - x_{k+1}}{\nabla f(x_{k+1}) - \nabla f(x_k)}
+ \scal{x_k - x_{k+1}}{\nabla f(x_k)}\\
&\geq f(x_{k+1}) - \frac{\delta}{\gamma_k\lambda_k}\norm{x_k - x_{k+1}}^2
+ \scal{x_k - x_{k+1}}{\nabla f(x_k)}.
\end{aligned}
\end{equation*}
hence \ref{lem:20160129a_b} holds. Now, suppose that \ref{lem:20160129a_b} holds.
Then, since $g(x_{k+1}) - g(x_k) \leq \lambda_k (g(y_k) - g(x_k))$, we have
\begin{align}
\label{eq:20160409a}
\nonumber(f+g)(x_{k+1}) - (f+g)(x_k)
&\leq g(x_{k+1}) - g(x_k)
+ \scal{x_{k+1} -x_k}{\nabla f(x_k)}
 + \frac{\delta}{ \gamma_k \lambda_k} \norm{x_{k+1} - x_k}^2\\
 & \leq \lambda_k \Big(g(y_{k}) - g(x_k)
+ \scal{y_{k} -x_k}{\nabla f(x_k)}
 + \frac{\delta}{ \gamma_k} \norm{y_{k} - x_k}^2\Big)\\
\nonumber & = (1-\delta)\lambda_k\big(g(y_{k}) - g(x_k)
+ \scal{y_{k} -x_k}{\nabla f(x_k)} \big)\\
\nonumber &\quad + \delta \lambda_k \bigg(g(y_{k}) - g(x_k)
+ \scal{y_{k} -x_k}{\nabla f(x_k)} 
 + \frac{1}{ \gamma_k} \norm{y_{k} - x_k}^2\bigg).
\end{align}
Therefore, recalling Lemma~\ref{lem:FBS_step}\ref{lem:FBS_stepi}, 
\ref{lem:20160129a_c} follows. 

\ref{lem:20160129aiii}:
It follows from   \ref{lem:20160129a_c} that
\begin{equation*}
- \lambda_k \big(g(y_{k}) - g(x_k) + \scal{y_{k} -x_k}{\nabla f(x_k)} \big)
\leq \frac{(f+g)(x_{k}) - (f+g)(x_{k+1})}{1 - \delta}
\end{equation*}
and hence, multiplying by $2 \gamma_k$,
the statement follows from Lemma~\ref{lem:FBS_step}\ref{lem:FBS_stepiii}.

\ref{lem:20160129aiv}:
It follows from \eqref{eq:20160202d},
by taking $x = x_k$.
\qquad\end{proof}


\begin{remark}
\normalfont
\label{eq:rmk20160409}
Another condition that is in between \ref{lem:20160129a_b} and
 \ref{lem:20160129a_c} of Proposition~\ref{lem:20160129a}  is the following
\begin{equation}
\label{eq:20160409b}
(f+g)(x_{k+1}) - (f+g)(x_k) \leq \sigma \lambda_k \Big( g(y_k) - g(x_k) 
+ \scal{y_k - x_k}{\nabla f(x_k)} + \frac{\beta}{\gamma_k} \norm{y_k - x_k}^2\Big),
\end{equation}
where $(\sigma,\beta) \in [0,1]^2$, $0 <  (1-\beta)\sigma<1$, and $1- \delta = (1-\beta)\sigma$. 
More precisely, 
since
\begin{align*}
g(y_k) - g(x_k) &+ \scal{y_k - x_k}{\nabla f(x_k)} 
+ \frac{1 - (1-\beta)\sigma}{\gamma_k} \norm{y_k - x_k}^2 \\
&= \sigma \Big( g(y_k) - g(x_k) + \scal{y_k - x_k}{\nabla f(x_k)}  
+ \frac{\beta}{\gamma_k} \norm{y_k - x_k}^2\Big)\\
&\qquad+ (1-\sigma)\Big( g(y_k) - g(x_k) + \scal{y_k - x_k}{\nabla f(x_k)}  
+ \frac{1}{\gamma_k} \norm{y_k - x_k}^2\Big)\\
&\leq \sigma \Big( g(y_k) - g(x_k) + \scal{y_k - x_k}{\nabla f(x_k)}  
+ \frac{\beta}{\gamma_k} \norm{y_k - x_k}^2\Big),
\end{align*}
if \ref{lem:20160129a_b} holds with $\delta = 1 - (1-\beta)\sigma$, then, 
it follows from \eqref{eq:20160409a} that \eqref{eq:20160409b} holds.
Moreover, it follows from  \eqref{eq:20160409a} and 
Lemma~\ref{lem:FBS_step}\ref{lem:FBS_stepi} that 
\eqref{eq:20160409b} $\Rightarrow$ \ref{lem:20160129a_c}
with $\delta = 1 - (1-\beta)\sigma$.
Note that  \eqref{eq:20160409b} includes \ref{lem:20160129a_c}
by choosing $\beta=0$ and $\sigma= \delta$.
Condition \eqref{eq:20160409b} 
(with $0<\sigma<1$ and $0 \leq \beta <1$) 
is at the basis of the Armijio line search proposed by Tseng and Yun in \cite{Tseng09},
which is also adopted in \cite{Bon15a} with $\beta\leq 1/2$. 
Finally, the second line search method in \cite{Cruz15}
 corresponds to \eqref{eq:20160409b} with $\sigma=1$, $\beta=\delta=1/2$, and $\lambda_k\equiv 1$, and hence it leads to values of $\lambda_k$ smaller than necessary:
by choosing $\delta$ close to $1$, a larger step 
along $y_k - x_k$ is obtained.
In view of Proposition~\ref{lem:20160129a}, 
Proposition~\ref{prop:20151220a}, Corollary~\ref{cor:20160202a} below, 
and the discussion above
we can claim that our convergence results hold also for the Armjio-type rule 
considered in \cite{Bon15a,Tseng09}.
\end{remark}

The following result treats the four line search methods in a unifying manner.

\begin{corollary}
\label{cor:20160202a}
Assume that {\bf H1} and {\bf H3} hold
and that $f$ is G\^ateaux differentiable on $\dom g$.
Let ${(x_k)}_{k \in \N}$ and ${(y_k)}_{k \in \N}$ be defined according to {\rm (VM-FBS)}
for some sequences $(\gamma_k)_{k \in \N}$ and $(\lambda_k)_{k \in \N}$.
 Suppose that
one of  
\ref{lem:20160129a_a}, \ref{lem:20160129a_b},
or \ref{lem:20160129a_c} in Proposition~\ref{lem:20160129a} is satisfied
in the metric $\scal{\cdot}{\cdot}_k$, for every $k \in \N$.
Then, the following hold.
\begin{enumerate}[{\rm(i)}]
\item\label{cor:20160202ai}
The sequence $\big{((f+g)(x_k)\big)}_{k \in \N}$ 
is decreasing.
\item \label{cor:20160202aii}
If $\sup_{k \in \N} \gamma_k<+\infty$ and 
$\inf_{k \in \N} (f+g)(x_k)>-\infty$, then
\begin{equation*}
\frac{1}{\sup_{k \in \N} \gamma_k}
\sum_{k \in \N} \norm{x_{k+1} - x_k}_k^2 \leq \sum_{k \in \N} 
\frac{\norm{x_{k+1} - x_{k}}_k^2}{\gamma_k} < +\infty.
\end{equation*}
\item\label{cor:20160202aiii}
If {\bf H4} holds,
then, for every $x \in \dom g$ and for every $k \in \N$,
\begin{multline*}
\norm{x_{k+1} - x}_{k+1}^2 \leq (1 + \eta_k) \norm{x_k - x}_k^2
+ 2 \gamma_k \lambda_k(1 + \eta_k) \big( (f+g)(x) - (f+g)(x_{k})\big)\\[0.5ex]
+ \frac{2 \gamma_k (1 + \eta_k)}{1 - \delta}\big( (f+g)(x_{k}) - (f+g)(x_{k+1}) \big).
\end{multline*}
\item\label{cor:20160202aiv}
If {\bf H5} holds,
then,
for every $x \in \dom g$ and for every $k \in \N$,
\begin{multline*}
\norm{x_{k+1} - x}^2 \leq\bigg(1+ \frac{\mu_k - \nu_k}{\nu}\bigg) \norm{x_k - x}^2
+ 2 \frac{\gamma_k \lambda_k}{\nu_k} \big( (f+g)(x) - (f+g)(x_{k})\big)\\[0.5ex]
+ \frac{2 \gamma_k }{\nu(1 - \delta)}\big( (f+g)(x_{k}) - (f+g)(x_{k+1}) \big).
\end{multline*}
\end{enumerate}
\end{corollary}
\begin{proof}
\ref{cor:20160202ai}:
Invoking Proposition~\ref{lem:20160129a}\ref{lem:20160129aiv}
for each metric $\scal{\cdot}{\cdot}_k$, we derive that,
for every $k \in \N$, $(f+g)(x_{k+1}) \leq (f+g)(x_k)$.

\ref{cor:20160202aii}:
Since 
 $\bar{\gamma} := \sup_{k \in \N} \gamma_k<+\infty$, we derive from 
Proposition~\ref{lem:20160129a}\ref{lem:20160129aiv}, applied 
with respect to  each metric
$\scal{\cdot}{\cdot}_k$, that, for every $k \in \N$
\begin{align*}
\frac{1-\delta}{\bar{\gamma}}\sum_{i=0}^k \norm{x_{i+1} - x_{i}}_i^2
&\leq (1-\delta)\sum_{i=0}^k \frac{\norm{x_{i+1} - x_{i}}_i^2}{\gamma_i }
\leq \sum_{i=0}^k \big( (f+g)(x_{i}) - (f+g)(x_{i+1}) \big) \\
&\leq (f+g)(x_0) - \inf_{k \in \N} (f+g)(x_k)<+\infty.
\end{align*}

\ref{cor:20160202aiii}:
Let $k \in \N$ and let $x \in \dom g$.
It is enough to note that, since $- (f+g)(x_k) \leq - (f+g)(x_{k+1})$, Proposition~\ref{lem:20160129a}\ref{lem:20160129aiii}, applied with respect to the metric $\scal{\cdot}{\cdot}_k$,  yields
\begin{equation}
\label{eq:20160204e}
\begin{aligned}
\norm{x_{k+1} - x}_k^2 \leq \norm{x_k - x}_k^2
&+ 2 \gamma_k \lambda_k \big( (f+g)(x) - (f+g)(x_{k+1})\big)\\[0.5ex]
&+ \frac{2 \gamma_k}{1 - \delta}\big( (f+g)(x_{k}) - (f+g)(x_{k+1}) \big).
\end{aligned}
\end{equation}
Hence, multiplying by $(1+\eta_k)$ and taking into account that
$\norm{\cdot}_{k+1}^2 \leq (1 + \eta_k) \norm{\cdot}^2_k$ the statement follows.

\ref{cor:20160202aiv}:
It follows from \eqref{eq:20160204e} and \eqref{eq:norms} that
\begin{equation}
\label{eq:20160204l}
\begin{aligned}
\nu_k \norm{x_{k+1} - x}^2 \leq \mu_k \norm{x_k - x}^2
&+ 2 \gamma_k \lambda_k \big( (f+g)(x) - (f+g)(x_{k+1})\big)\\[0.5ex]
&+ \frac{2 \gamma_k}{1 - \delta}\big( (f+g)(x_{k}) - (f+g)(x_{k+1}) \big).
\end{aligned}
\end{equation}
Hence, dividing \eqref{eq:20160204l} by $\nu_k$, and noting
that $(f+g)(x_{k}) - (f+g)(x_{k+1})\geq 0$, that $\nu \leq \nu_k$, and $\mu_k/\nu_k = 1 + (\mu_k - \nu_k)/\nu_k \leq 1 + (\mu_k - \nu_k)/\nu$,
the statement follows.
\qquad\end{proof}

\begin{remark}
\normalfont
\label{rmk:20160202a}
Under the hypotheses of Corollary~\ref{cor:20160202a}, 
if $(\gamma_k)_{k \in \N}$ is bounded and  $\inf_{k \in \N} (f+g)(x_k)>-\infty$,
then the following hold:
\begin{enumerate}[(i)]
\item 
If {\bf H4} is satisfied, then setting $\bar{\gamma} = \sup_{k \in \N} \gamma_k$ and $\bar{\eta} = \sup_{k \in \N} \eta_k$, condition 
  \ref{prop:genconv_b} 
in Proposition~\ref{prop:genconv}
is fulfilled with $h = f+g$,
$(\alpha_k)_{k \in \N} = (\gamma_k \lambda_k (1+ \eta_k))_{k \in \N}$, 
${(\abs{\cdot}_k)}_{k \in \N} = {(\norm{\cdot}_k)}_{k \in \N}$, and
\begin{equation*}
(\varepsilon_k)_{k \in \N} = 2 \bar{\gamma}(1 + \bar{\eta})/(1-\delta) 
{ \big( (f+g)(x_k) - (f+g)(x_{k+1}) \big)}_{k \in \N}.
\end{equation*}
\item\label{rmk:20160202aii}
If {\bf H5} satisfied, then setting $\bar{\gamma} = \sup_{k \in \N} \gamma_k$, condition 
  \ref{prop:genconv_b} 
in Proposition~\ref{prop:genconv}
is fulfilled with $h = f+g$,
$(\alpha_k)_{k \in \N} = (\gamma_k \lambda_k/\nu_k)_{k \in \N}$, 
${(\eta_k)}_{k \in \N} = \nu^{-1}{(\mu_k - \nu_k)}_{k \in \N}$, 
${(\abs{\cdot}_k)}_{k \in \N} \equiv \norm{\cdot}$, and
\begin{equation*}
(\varepsilon_k)_{k \in \N} = 2 \bar{\gamma}/(\nu(1-\delta))
{ \big( (f+g)(x_k) - (f+g)(x_{k+1}) \big)}_{k \in \N}.
\end{equation*}
\end{enumerate}
\end{remark}

To finish our convergence analysis it remains to
verify \ref{prop:genconv_c} in Proposition~\ref{prop:genconv}.

\begin{proposition}
\label{prop:20151220a}
Assume that {\bf H1}--{\bf H3} hold
and define ${(x_k)}_{k \in \N}$ and ${(y_k)}_{k \in \N}$ as in  {\rm (VM-FBS)}
with some $(\gamma_k)_{k \in \N}$ and $(\lambda_k)_{k \in \N}$.
Let $\delta>0$ and
consider the following properties:

\begin{enumerate}[$(a^\prime)$]
\item
\label{prop:20151220a_b} 
$\displaystyle (\exists\, \theta \in \left]0,1\right[)(\exists\, \sigma \in \RPP) (\forall\, k \in \N)\quad  
\gamma_k <\sigma\ \Rightarrow\\[1ex]
\phantom{aa}
\big\lVert \nabla^k f\big(J_k(x_k,\gamma_k/\theta, \lambda_k)\big) - \nabla^k f(x_k) \big\rVert_k
> \frac{\delta \theta}{\gamma_k \lambda_k} 
\norm{J_k(x_k,\gamma_k/\theta,\lambda_k) - x_k}_k$.\\[-0.5ex]
\item
\label{prop:20151220b_b} 
$\displaystyle (\exists\, \theta \in \left]0,1\right[)(\exists\, \sigma \in \RPP) (\forall\, k \in \N)\quad  
\gamma_k <\sigma\ \Rightarrow\\[1ex]
f\big(J_k(x_k, \gamma_k/\theta,\lambda_k)\big) - f(x_k) 
- \big\langle J_k(x_k, \gamma_k/\theta,\lambda_k) -x_k 
\,\big\vert\, \nabla f(x_k)\big\rangle
\!>\! \frac{\delta \theta}{\gamma_k \lambda_k} 
\norm{J_k(x_k,\gamma_k/\theta,\lambda_k) - x_k}_k^2$.
\end{enumerate}
\begin{enumerate}[$(b^{\prime\prime})$]
\item
\label{prop:20151220b_b2}
$\displaystyle (\exists\, \theta \in \left]0,1\right[)(\exists\, \sigma \in \left]0,1\right])(\forall\, k \in \N)
\ \lambda_k\leq \sigma\theta\ \Rightarrow\\[1ex] 
f\big(J_k(x_k,\gamma_k,\lambda_k/\theta)\big) - f(x_k)
- \big\langle J_k (x_k, \gamma_k, \lambda_k/\theta) -x_k 
\,\big\vert\, \nabla f(x_k)\big\rangle
> \frac{\delta \theta}{\gamma_k \lambda_k} 
\norm{J_k(x_k,\gamma_k,\lambda_k/\theta) - x_k}_k^2$.
\end{enumerate}
\begin{enumerate}[$(c^{\prime})$]
\item 
\label{prop:20151220b_c}
$\displaystyle (\exists\, \theta \in \left]0,1\right[)(\exists\, \sigma \in \left]0,1\right])
(\forall\, k \in \N)
\ \gamma_k<\sigma \Rightarrow \\[1.5ex]
(f+g)\big(J_k (x_k, \gamma_k/\theta, \lambda_k)\big) - (f+g)(x_k) \\[1.5ex] 
\phantom{xxxxxxxxxxxxxxx}> (1-\delta) \lambda_k\big( g\big(J_k(x_k, \gamma_k/\theta, 1)\big) - g(x_k)
+ \big\langle J_k(x_k, \gamma_k/\theta, 1) -x_k 
\,\big\vert\, \nabla f(x_k)\big\rangle \big)$.
\end{enumerate}
\begin{enumerate}[$(c^{\prime\prime})$]
\item 
\label{prop:20151220b_c2}
$\displaystyle (\exists\, \theta \in \left]0,1\right[)(\exists\, \sigma \in \left]0,1\right])(\forall\, k \in \N)
\ \lambda_k\leq \sigma\theta\ \Rightarrow\\[1.5ex]
(f+g)\big(J_k(x_k,\gamma_k, \lambda_k/\theta)\big) - (f+g)(x_k) 
> (1-\delta)(\lambda_k/\delta)
\big( g(y_k) - g(x_k)
+ \big\langle y_k -x_k 
\,\big\vert\, \nabla f(x_k)\big\rangle \big)$.
\end{enumerate}
Then the following hold.
\begin{enumerate}[{\rm(i)}]
\item\label{prop:20151220aibis} 
\ref{prop:20151220b_c} $\Rightarrow$ \ref{prop:20151220b_b}
 $\Rightarrow$ \ref{prop:20151220a_b}
and \ref{prop:20151220b_c2} $\Rightarrow$ \ref{prop:20151220b_b2}.
\item\label{prop:20151220aii}
Suppose that $\big((f+g)(x_k)\big)_{k \in \N}$ is decreasing, $\norm{x_{k+1} - x_k}_k \to 0$,  $\sup_{k \in \N} \gamma_k<+\infty$, and that
either  $\inf_{k \in \N} \lambda_k>0$ and 
\ref{prop:20151220a_b} hold, or that
$\inf_{k \in \N} \gamma_k>0$ and \ref{prop:20151220b_b2}  hold.
Then, there exists  
$(v_k)_{k\in \N}\in \hh^{\N}$ such that,  $\forall\, k \in \N$, $v_k \in \partial (f+g)(y_k)$
and for every weakly convergent subsequence $(x_{n_k})_{k \in \N}$
of $(x_k)_{k \in \N}$, $y_{k_n} - x_{k_n}\to 0$ and $v_{n_k}\to 0$.
\end{enumerate}
\end{proposition}
\begin{proof}
\ref{prop:20151220aibis}:
It follows from Proposition~\ref{lem:20160129a}\ref{lem:20160129ai},
applied for each metric $\scal{\cdot}{\cdot}_k$,
by using a simple contradiction argument.

\ref{prop:20151220aii}:
Let, for every $k \in \N$, 
$y_{k} = \prox^k_{\gamma_{k} g} 
( x_{k} - \gamma_{k} \nabla^k f(x_{k}))$. Then,
for every $k \in \N$,
$(x_{k} - y_k)/\gamma_{k}+ \nabla^k f(y_k) - \nabla^k f(x_{k}) 
\in \partial^k (f+g)(y_k)$,
hence, by \eqref{eq:metric_subdiff},
\begin{equation}
\label{eq:20151227f}
(\forall\, k \in \N)\quad v_k: = W_k \frac{x_{k} - y_{k}}{\gamma_{k}} 
+ \nabla f(y_{k}) - \nabla f(x_{k}) 
\in \partial (f+g)(y_{k}).
\end{equation}
Let $(x_{n_k})_{k \in \N}$ be a subsequence of $(x_k)_{k \in \N}$
such that $x_{n_k} \rightharpoonup \bar{x}$ for some $\bar{x} \in \hh$. We note that,
since $\big((f+g)(x_k)\big)_{k \in \N}$ is decreasing  $\lim (f+g)(x_{n_k}) = 
\inf_{k} (f+g)(x_{n_k}) <+\infty$, hence it follows from the lower semicontinuity of $f+g$
 that $\bar{x} \in \dom (f+g) \subset \dom g$.\\
Suppose first that  \ref{prop:20151220a_b} is satisfied
and that $\inf_{k \in \N} \lambda_k>0$.
Let, for every $k \in \N$, $\tilde{\gamma}_k = \gamma_k/\theta$ and 
$\tilde{x}_{k} = J_k(x_k,\tilde{\gamma}_k, \lambda_k)$. Then, for every $k \in \N$,
by Lemma~\ref{lem:Jproperties}\ref{lem:Jpropertiesi} 
\begin{equation}
\label{eq:20151227a}
\norm{J_k(x_k,\tilde{\gamma}_k,\lambda_k) - x_k}_k 
\leq \frac{\tilde{\gamma}_k}{\gamma_k} \norm{J_k(x_k, \gamma_k,\lambda_k) - x_k}_k 
= \frac{1}{\theta}\norm{x_{k+1} - x_k}_k
\end{equation}
and  $\norm{x_{k+1} - x_k}_k = \norm{J_k(x_k, \gamma_k,\lambda_k) - x_k}_k 
\leq \norm{J_k(x_k, \tilde{\gamma}_k,\lambda_k ) - x_k}_k$;
hence, recalling also \eqref{eq:norms},
\begin{equation}
\label{eq:20151227b}
\sqrt{\nu} \frac{\norm{y_k - x_k}}{\gamma_k } 
\leq \frac{\norm{y_k - x_k}_k}{\gamma_k } 
=\frac{\norm{x_{k+1} - x_k}_k}{\gamma_k \lambda_k} 
\leq \frac{\norm{\tilde{x}_{k} - x_k}_k}{\gamma_k \lambda_k}.
\end{equation}
Moreover, since \ref{prop:20151220a_b} is satisfied, then
\begin{equation}
\label{eq:20151229e}
(\forall\, k \in \N)\  
\max\Big\{ \big\lVert \nabla^k f(\tilde{x}_{k}) - \nabla^k f(x_k)\big\rVert_k, 
\frac{\delta \theta}{\sigma} \frac{\norm{\tilde{x}_{k} - x_k}_k}{\lambda_k} \Big\}
\geq \delta \theta\frac{\norm{\tilde{x}_{k} - x_k}_k}{\lambda_k \gamma_k}.
\end{equation}
Now, since $\norm{x_{k+1} - x_{k}}_k \to 0$, we derive from \eqref{eq:20151227a}
and \eqref{eq:norms},
that $\tilde{x}_{k} - x_k \to 0$ and since 
$x_{n_k} \rightharpoonup \bar{x}$,
it follows from Lemma~\ref{l:20160103a}\ref{l:20160103ai} that  
$\norm{\nabla f(\tilde{x}_{n_k}) - \nabla f(x_{n_k})} \to 0$,
hence, by Fact~\ref{p:20160125f} and the fact that $\nu \leq \nu_k$, 
we have
\begin{equation}
\label{eq:20160204g}
\big\lVert\nabla^k f(\tilde{x}_{n_k}) - \nabla^k f(x_{n_k})\big\rVert_k \to 0.
\end{equation}
Moreover, 
\begin{equation}
\label{eq:20160125h}
\frac{\norm{\tilde{x}_{k} - x_k}_k}{\lambda_k} \leq 
\frac{\norm{\tilde{x}_{k} - x_k}_k}{\inf_{k \in \N} \lambda_k} \to 0.
\end{equation}
Thus, \eqref{eq:20160204g}, \eqref{eq:20160125h}, 
and \eqref{eq:20151229e} imply
$\norm{\tilde{x}_{n_k} - x_{n_k}}_k/(\gamma_{n_k} \lambda_{n_k}) \to 0$
and so, by 
\eqref{eq:20151227b},
\begin{equation}
\label{eq:20160109g}
\frac{\norm{y_{n_k} - x_{n_k}}}{\sup_{k \in \N}\gamma_{k}}
\leq \frac{\norm{y_{n_k} - x_{n_k}}}{\gamma_{n_k}}  \to 0.
\end{equation}
Finally, since $y_{n_k} - x_{n_k} \to 0$ and $x_{n_k} \rightharpoonup \bar{x}$,
by another application of Lemma~\ref{l:20160103a}\ref{l:20160103ai}, we have that
$\norm{\nabla f(y_{n_k}) - \nabla f(x_{n_k})} \to 0$. This 
 together with \eqref{eq:20160109g}, \eqref{eq:20151227f}
 and the fact that $\sup_{k \in \N} \norm{W_k} \leq \mu$,
gives $v_{n_k} \to 0$.\\
Now suppose that 
 \ref{prop:20151220b_b2} is satisfied and that
$\inf_{k \in \N} \gamma_k>0$.
Set, for every $k \in \N$, $\tilde{\lambda}_k = \lambda_k/\theta$ and 
$\tilde{x}_{k} = J_k(x_k,\gamma_k,\tilde{\lambda}_k)$. 
Then, 
by Lemma~\ref{lem:Jproperties}\ref{lem:Jpropertiesi2},
for every $k \in \N$,
\begin{equation}
\label{eq:20160106q}
\lVert J_k(x_k,\gamma_k,\tilde{\lambda}_k) - x_k \rVert_k
= \frac{\tilde{\lambda}_k}{\lambda_k} \norm{J_k(x_k,\gamma_k, \lambda_k) - x_k}_k 
= \frac{1}{\theta}\norm{x_{k+1} - x_k}_k
\end{equation}
and  $\norm{x_{k+1} - x_k}_k = \norm{J_k(x_k, \gamma_k, \lambda_k) - x_k}_k 
\leq \lVert J_k(x_k, \gamma_k, \tilde{\lambda}_k ) - x_k \rVert_k$;
hence
\begin{equation}
\label{eq:20160106o}
\sqrt{\nu} \frac{\norm{y_k - x_k}}{\gamma_k } 
\leq \frac{\norm{y_k - x_k}_k}{\gamma_k } 
=\frac{\norm{x_{k+1} - x_k}_k}{\gamma_k \lambda_k} 
\leq \frac{\norm{\tilde{x}_{k} - x_k}_k}{\gamma_k \lambda_k}.
\end{equation}
Moreover, since \ref{prop:20151220b_b2} is satisfied,
\begin{equation}
\label{eq:20160106p}
(\forall\, k \in \N)\  
\max\bigg\{\frac{\lvert f(\tilde{x}_{k}) - f(x_k) - \scal{\tilde{x}_{k} -x_k}{\nabla f(x_k)}_k\rvert }
{\norm{\tilde{x}_{k} - x_k}_k}, \frac{ \delta}{ \sigma } 
\frac{\norm{\tilde{x}_{k} - x_k}_k}{\gamma_k} \bigg\}
> \delta \theta\frac{\norm{\tilde{x}_{k} - x_k}_k}{\gamma_k \lambda_k}.
\end{equation}
Now, since $\norm{x_{k+1} - x_{k}}_k \to 0$, we derive from \eqref{eq:20160106q}
and \eqref{eq:norms},
that $\tilde{x}_{k} - x_k \to 0$ and,
since $x_{n_k} \rightharpoonup \bar{x}$,
it follows from Lemma~\ref{l:20160103a}\ref{l:20160103ai} that 
\begin{equation*}
\frac{\lvert f(\tilde{x}_{n_k}) - f(x_{n_k}) - \scal{\tilde{x}_{n_k} -x_{n_k}}{\nabla f(x_{n_k})}\rvert}
{\norm{\tilde{x}_{n_k} - x_{n_k}}} \to 0.
\end{equation*}
Since $\scal{\tilde{x}_{n_k} -x_{n_k}}{\nabla f(x_{n_k})} 
=  \scal{\tilde{x}_{n_k} -x_{n_k}}{\nabla^k f(x_{n_k})}_k$ and 
$\nu^{1/2}/\norm{\tilde{x}_{n_k} -x_{n_k}}_k \leq 1/\norm{\tilde{x}_{n_k} -x_{n_k}}$,
we have
\begin{equation}
\label{eq:20160106r}
\frac{\lvert f(\tilde{x}_{n_k}) - f(x_{n_k}) 
- \scal{\tilde{x}_{n_k} -x_{n_k}}{\nabla^k f(x_{n_k})}_k\rvert}
{\norm{\tilde{x}_{n_k} - x_{n_k}}_k} \to 0.
\end{equation}
Moreover,
\begin{equation}
\label{eq:20160125n}
\frac{\norm{\tilde{x}_{k} - x_k}_k}{\gamma_k} 
\leq \frac{\norm{\tilde{x}_{k} - x_k}_k}{\inf_{k \in \N} \gamma_k} \to 0.
\end{equation}
Now it follows from \eqref{eq:20160106p}, \eqref{eq:20160106r},
\eqref{eq:20160125n}, and 
\eqref{eq:20160106o}
 that
\begin{equation}
\label{eq:20160106t}
\frac{\norm{y_{n_k} - x_{n_k}}}{\sup_{k \in \N} \gamma_k} 
\leq \frac{\norm{y_{n_k} - x_{n_k}}}{\gamma_{n_k}} \to 0
\end{equation}
and, since
$x_{n_k} \rightharpoonup \bar{x}$, we derive again from 
 Lemma~\ref{l:20160103a}\ref{l:20160103ai} that
\begin{equation}
\label{eq:20160109e}
\norm{\nabla f(y_{n_k}) - \nabla f(x_{n_k})} \to 0.
\end{equation}
Therefore, since $\sup_{k \in \N} \norm{W_k} \leq \mu$, 
 \eqref{eq:20151227f}, \eqref{eq:20160106t}, and
\eqref{eq:20160109e}, yields
$v_{n_k} \to 0$.
\qquad\end{proof}

We finally present the main convergence theorem.

\begin{theorem}
\label{thm:convergence}
Assume that {\bf H1}--{\bf H3} hold and that either of the two {\bf H4} or {\bf H5} hold.
Let $(x_{k})_{k \in \N}$ and $(y_{k})_{k \in \N}$ 
be generated by algorithm {\rm (VM-FBS)} using one of the line search procedures
\ref{LS1}--\ref{LS3} for determining the parameters
$(\gamma_k)_{k \in \N}$ and $(\lambda_k)_{k \in \N}$.
Set $S_* = \argmin_\hh (f+g)$.
Then the following hold.
\begin{enumerate}[{\rm(i)}]
\item\label{thm:convergencei0} ${((f+g)(x_k))}_{k \in \N}$ is decreasing.
\item\label{thm:convergenceii01}  If $\inf_{k \in \N}(f+g)(x_k)>-\infty$, then
$\sum_{k \in \N} \norm{x_{k+1} - x_k}^2<+\infty$.
\item\label{thm:convergencei} 
Suppose that $S_* \neq \varnothing$. Then 
\begin{enumerate}[$(a)$]
\item  $(x_{k})_{k \in \N}$ and  $(y_{k})_{k \in \N}$ weakly converge 
to the same point in $S_*$.
\item $(f+g)(y_k) \to \inf_{\hh}(f+g)$.
\item If $\sum_{k \in \N} \gamma_k\lambda_k = +\infty$,
 then $(f+g)(x_k) \to \inf_{\hh}(f+g)$.
\item\label{thm:convergenceid}  If  $\inf_{k \in \N} \gamma_k\lambda_k> 0$, then
$\big( (f+g)(x_k) - \inf_{\hh}(f+g)\big) = o(1/k)$.
\end{enumerate}
\item\label{thm:convergenceii}  
If $S_* = \varnothing$, then 
 $\norm{x_k}\to +\infty$ and $(f+g)(x_k) \to \inf_\hh (f+g)$.
\end{enumerate}
\end{theorem}
\begin{proof}
Let $h=f+g$.
For any proposed line search method,
one of the properties \ref{lem:20160129a_a}, \ref{lem:20160129a_b},
or \ref{lem:20160129a_c}  in Proposition~\ref{lem:20160129a}
is satisfied for every $k \in \N$ with respect to $\scal{\cdot}{\cdot}_k$, and
the corresponding property
\ref{prop:20151220a_b}, \ref{prop:20151220b_b}-\ref{prop:20151220b_b2},
or \ref{prop:20151220b_c2} 
in Proposition~\ref{prop:20151220a} is satisfied too. Therefore, 
by Corollary~\ref{cor:20160202a}\ref{cor:20160202ai}, 
$((f+g)(x_k))_{k \in \N}$ is decreasing and
\ref{prop:genconv_a}
in Proposition~\ref{prop:genconv}
is fulfilled.
Moreover, 
Remark~\ref{rmk:20160202a}
ensures that if $\inf_{k \in \N}(f+g)(x_k)>-\infty$,
then
  \ref{prop:genconv_b} 
in Proposition~\ref{prop:genconv}
is fulfilled for 
$(\alpha_k)_{k \in \N} = {(\beta_k\gamma_k \lambda_k)}_{k \in \N}$
(with $(\beta_k)_{k \in \N} \in \RPP^{\N}$ such that
$0<\inf_{k \in \N} \beta_k \leq \sup_{k \in \N} \beta_k<+\infty$).
Finally, 
Corollary~\ref{cor:20160202a}\ref{cor:20160202aii}
and Proposition~\ref{prop:20151220a}\ref{prop:20151220aii}
implies that if $\inf_{k \in \N}(f+g)(x_k)>-\infty$,
then \ref{prop:genconv_c} 
in Proposition~\ref{prop:genconv}
is  fulfilled too.
Then
the statements follow from Theorem~\ref{thm:convergence02}.
\qquad\end{proof}

\begin{remark}
\normalfont
In the proof of Theorem~\ref{thm:convergence}, we showed that when {\bf H5}
is in force, we apply Proposition~\ref{prop:genconv} 
using the metric of the Hilbert space $\hh$, instead of
 the variable metrics
${(\scal{\cdot}{\cdot}_k)}_{k \in \N}$ of {\bf H3}
(see Remark~\ref{rmk:20160202a}\ref{rmk:20160202aii}). 
This is why if we assume {\bf H5} we do not require
the convergence of the $W_k$'s.
\end{remark}

\begin{corollary}
\label{cor:convergence}
Assume that {\bf H1}--{\bf H3} hold and that either {\bf H4} or {\bf H5} hold.
Let $(x_{k})_{k \in \N}$   be generated by algorithm {\rm (VM-FBS)} using either 
 \ref{LS1} or \ref{LS3}
with ${(\lambda_k)}_{k \in \N} \equiv 1$ (no relaxation).
Set $S_* = \argmin_\hh (f+g)$.
Then the following hold.
\begin{enumerate}[{\rm(i)}]
\item\label{cor:convergence0i} ${((f+g)(x_k))}_{k \in \N}$ is decreasing and
 $(f+g)(x_k) \to \inf_{\hh}(f+g)$.
\item\label{cor:convergence0ii}  If $\inf_\hh(f+g)>-\infty$, then
$\sum_{k \in \N} \norm{x_{k+1} - x_k}^2<+\infty$.
\item\label{cor:convergencei} 
Suppose that $S_* \neq \varnothing$. Then 
\begin{enumerate}[$(a)$]
\item\label{cor:convergenceia}  $(x_{k})_{k \in \N}$ weakly converges to a point in $S_*$.
\item\label{cor:convergenceib}  If  $\inf_{k \in \N} \gamma_k> 0$, then
$\big( (f+g)(x_k) - \inf_{\hh}(f+g)\big) = o(1/k)$.
\end{enumerate}
\item\label{cor:convergenceii}  
If $S_* = \varnothing$, then 
 $\norm{x_k}\to +\infty$.
\end{enumerate}
\end{corollary}

\begin{remark}
\normalfont
\label{rmk:20160114b}
In \cite[Theorem~4.2 and Theorem~4.3]{Cruz15} the same line search \ref{LS3}
is studied for the stationary (metric) forward-backward algorithm.
However, even in this setting, the corresponding results given in Theorem~\ref{thm:convergence}
are more general and stronger. More precisely
in \cite[Method 1]{Cruz15}: $a)$
no relaxation is allowed, that is $\lambda_k \equiv 1$;
$b)$  $\delta$ is required to be strictly less than $1/2$ (this halves
the stepsizes compared with those of \ref{LS3} --- see also 
Remark~\ref{rmk:20160110a}\ref{rmk:20160110aii});
$c)$ $\nabla f$ is required to be uniformly continuous on any bounded subsets of $\dom g$
and to map bounded sets into bounded sets; 
$d)$ the little-$o$ rate of convergence is provided
only for $\hh$ finite-dimensional.
\end{remark}

Proposition~\ref{prop:20151220a} 
and Theorem~\ref{thm:convergence02}
allow also to obtain new convergence results for the standard variable
metric forward-backward algorithm (without backtracking) \cite{Comb2014} 
by requiring the Lipschitz continuity of the gradient on the domain of $g$ only.

\begin{theorem}
\label{thm:FBS}
Assume that {\bf H1} and {\bf H3} hold and that either {\bf H4} or {\bf H5} hold.
Suppose 
$f$ is G\^ateaux differentiable on $\dom g$ and $\nabla f$
is $L$-Lipschitz continuous on $\dom g$ for some $L \in \RP$.
Define ${(x_k)}_{k \in \N}$ and ${(y_k)}_{k \in \N}$ as in  {\rm (VM-FBS)}
and suppose  ${(\lambda_k)}_{k \in \N}$
and ${(\gamma_k)}_{k\in \N}$ are chosen a priori in a such way that
 $\inf_{k \in \N} \lambda_k>0$, 
$\inf_{k \in \N} \gamma_k>0$, and $\sup_{k \in \N} \gamma_k \lambda_k/\nu_k<2/L$.
Then the conclusions of Theorem~\ref{thm:convergence} hold.
\end{theorem}
\begin{proof}
Assumption  {\bf H2} is fulfilled too. 
Set $h = f+g$ and $(\alpha_k)_{k \in \N} = (\beta_k\gamma_k \lambda_k)_{k \in \N}$,
where $\beta_k$ is equal to $(1+\eta_k)$ if {\bf H4} holds,
and $1/\nu_k$ if {\bf H5} holds.
Since $\sup_{k \in \N} \gamma_k \lambda_k/\nu_k < 2/L$, 
there exists $\delta \in \left]0,1\right[$ 
such that $\sup_{k \in \N} \gamma_k \lambda_k/\nu_k \leq 2 \delta/L$.
Hence, 
since $\nabla f$ is 
$L$-Lipschitz continuous on $\dom g$ and $\{x_k \,\vert\,k\in \N\} \subset \dom g$, 
Fact~\ref{fact_desclem} 
ensures that 
 \begin{equation*}
(\forall\, k \in \N)\qquad f(x_{k+1}) - f(x_k) 
- \scal{x_{k+1} - x_k}{ \nabla f(x_k)} \leq \frac{L}{2} \norm{x_{k+1} - x_k}^2.
\end{equation*}
Since $\scal{x_{k+1} - x_k}{\nabla f(x_k)} = \scal{x_{k+1} - x_k}{\nabla^k f(x_k)}_k$
and $L/(2 \nu_k) \leq \delta/(\gamma_k \lambda_k)$, 
by \eqref{eq:norms}, we have
\begin{equation*}
(\forall\, k \in \N)\qquad
\frac{L}{2} \norm{x_{k+1} - x_k}^2 \leq \frac{L}{2 \nu_k} \norm{x_{k+1} - x_k}^2_k
\leq \frac{\delta}{\gamma_k \lambda_k} \norm{x_{k+1} - x_k}^2_k.
\end{equation*}
Thus,  \ref{lem:20160129a_b} in Proposition~\ref{lem:20160129a}
is satisfied
for every $k \in \N$ with respect to $\scal{\cdot}{\cdot}_k$.
Moreover, since $\inf_{k \in \N} \gamma_k>0$, condition 
\ref{prop:20151220b_b} in Proposition~\ref{prop:20151220a}
is also trivially satisfied.
Then, by 
Corollary~\ref{cor:20160202a}, Remark~\ref{rmk:20160202a},
and Proposition~\ref{prop:20151220a},
we have that 
 \ref{prop:genconv_a} in Proposition~\ref{prop:genconv} is fulfilled
and that, if $\inf_{k \in \N} h(x_k) > -\infty$, 
conditions  \ref{prop:genconv_b} 
and \ref{prop:genconv_c}
in Proposition~\ref{prop:genconv}
are fulfilled too. Then the statements follow
from Theorem~\ref{thm:convergence02}.
\end{proof}

\begin{remark}\ 
\normalfont
\begin{enumerate}[(i)]
\item Theorem~\ref{thm:FBS} provides a worst case rate of convergence which is new.
\item Theorem~\ref{thm:FBS} shows that the gradient descent stepsize 
parameter $\gamma_k$, the relaxation parameter $\lambda_k$,
and the minimum eigenvalues of the metric $\scal{\cdot}{\cdot}_k$
are linked together by the condition $\sup_{k \in \N} \gamma_k \lambda_k/\nu_k <2/L$.
Thus, if $\lambda_k$ is reduced (or $\nu_k$ is increased), one is allowed to enlarge
the stepsize $\gamma_k$, which may therefore exceed $2/L$.
This result complements that in \cite{Comb15} (for stationary metrics), where 
the parameters $\gamma_k$'s and $\lambda_k$'s 
appear linked too, but in that case,
it is the relaxation parameter that can go beyond the usual bound $1$.
\item In Theorem~\ref{thm:FBS} $\nabla f$ is not required 
to have full domain. Thus, the above result is not covered by the
 convergence theory developed in \cite{Comb2014,Smms05,Comb15}, 
since there a full domain of the gradient is required by
 the application of the Baillon-Haddad  theorem.
This aspect has been also noted in \cite{Cho14}.
\end{enumerate}
\end{remark}

In view of Theorem~\ref{thm:convergence}\ref{thm:convergencei}\ref{thm:convergenceid}
it is important to know conditions that guarantees that 
$(\gamma_k \lambda_k)_{k \in \N}$ remains bounded away from zero,
since in such case  (VM-FBS) has $o(1/k)$ rate of convergence
in function values. We now addresses this issue.

\begin{proposition}
\label{p:20151228a}
Assume that {\bf H3}  and  either {\bf H4} or {\bf H5} hold.
Let $f, g \in \Gamma_0(\hh)$ with $\dom g \subset \dom f$
and suppose that $f$ is G\^ateaux differentiable on $\dom g$.
\begin{enumerate}[{\rm(i)}]
\item\label{p:20151228ai} 
Suppose that $\nabla f$ is globally Lipschitz continuous on $\dom g$ 
with constant $L$ and that 
$(\gamma_k)_{k \in \N}$, and ${(\lambda_k)}_{k \in \N}$
are generated through algorithm {\rm (VM-FBS)} using any of 
the line seaches \ref{LS1}--\ref{LS3}.
 Then
\begin{itemize}
\item for \ref{LS1}, we have
$\inf_{k \in \N} \gamma_k \geq \min\{\bar{\gamma}, 
2 \delta \theta \nu/(L \sup_{k \in \N} \lambda_k\})\}$;
\item for \ref{LS2} or \ref{LS4}, 
we have
$\inf_{k \in \N} \lambda_k \geq \min\{\bar{\lambda}, 
2\delta \theta \nu/(L \sup_{k \in \N} \gamma_k)\}$.
\item for \ref{LS3}, we have
 $\inf_{k \in \N} \gamma_k \geq \min\{\bar{\gamma}, 
 \delta \theta \nu/(L \sup_{k \in \N} \lambda_k\})\}$.
\end{itemize}
\item\label{p:20151228aii}  
Suppose that $S_*=\argmin_\hh (f+g) \neq \varnothing$ and that
$\nabla f$ is Lipschitz continuous on any weakly compact subset of $\dom g$.
Let $(\gamma_k)_{k \in \N}$, and ${(\lambda_k)}_{k \in \N}$
be defined according to algorithm {\rm (VM-FBS)} using any of
the line search procedures \ref{LS1}--\ref{LS3}.
Then $\inf_{k \in \N}\gamma_k >0$ and $\inf_{k \in \N} \lambda_k>0$.
\end{enumerate}
\end{proposition}
\begin{proof} We first remark that
in both statements, {\bf H1} and {\bf H2} are fulfilled.

\ref{p:20151228ai}:
First we  consider the case that  
$(x_k)_{k \in \N}$, $(\gamma_k)_{k \in \N}$, and ${(\lambda_k)}_{k \in \N}$
are generated using either \ref{LS1}, \ref{LS2} or \ref{LS4}.
Let $k \in \N$. Since $\nabla f$ is $L$-Lipschitz continuous on $\dom g$,
we derive from the descent lemma (Fact~\ref{fact_desclem}) that
for every $\gamma>0$
\begin{equation}
\label{eq:20160121a}
f(J(x_k,\gamma, \lambda_k)) - f(x_k) - \scal{J(x_k,\gamma, \lambda_k) -x_k}{\nabla f(x_k)}
 \!\leq\! \frac{L}{2} \norm{J_k(x_k,\gamma, \lambda_k) \!-\! x_k}^2
\end{equation}
and for every $\lambda \in \left]0,1\right]$,
\begin{equation}
\label{eq:20160126a}
 f(J_k(x_k,\gamma_k, \lambda)) - f(x_k) 
- \scal{J(x_k,\gamma_k, \lambda) \!-\!x_k}{\nabla f(x_k)}
 \!\leq\! \frac{L}{2} \norm{J_k(x_k,\gamma_k, \lambda) \!-\! x_k}^2\!.
\end{equation}
Moreover, again by \eqref{eq:norms}, 
\begin{equation}
\label{eq:20160204h}
(\forall\, (\gamma,\lambda) \in \RPP\times\left]0,1\right])\quad 
\frac{L}{2} \norm{J_k(x_k,\gamma, \lambda) - x_k}^2 
\leq \frac{L}{2 \nu_k} \norm{J_k(x_k,\gamma, \lambda) - x_k}^2_k.
\end{equation}
Define
\vspace{-1ex}
\begin{align}
\label{eq:20151228f}
\bar{\gamma}_k &= \max\big\{ \gamma \in \RPP \,\vert\, 
(\exists\, i \in \N)(\gamma = \bar{\gamma} \theta^i)\  L 
\leq 2\nu_k\delta/(\gamma \lambda_k) \big\},\\
\label{eq:20160126b}
\bar{\lambda}_k &= \max\big\{ \lambda \in \left]0,1\right] \,\vert\, 
(\exists\, i \in \N)(\lambda = \bar{\lambda} \theta^i)\  L 
\leq 2\delta \nu_k/(\gamma_k\lambda) \big\}.
\end{align}
It follows from \eqref{eq:20160121a}, \eqref{eq:20160204h}, and the definition of $\gamma_k$
in \ref{LS1}
that $\gamma_k \geq \bar{\gamma}_k$.
Moreover, by \eqref{eq:20151228f}, we have that if  $\bar{\gamma}_k < \bar{\gamma}$, 
then 
$L> 2 \delta \theta \nu_k/(\bar{\gamma}_k \lambda_k)$, hence
$\bar{\gamma}_k> 2 \delta \theta \nu_k/(L \lambda_k)$. 
Therefore $\bar{\gamma}_k \geq \min\{\bar{\gamma}, 2\delta \theta \nu/(L \sup_{k \in \N} \lambda_k\}$.
Similarly, it follows from \eqref{eq:20160126a}, \eqref{eq:20160204h},
and  the definition of $\lambda_k$
in \ref{LS2} or in \ref{LS4}, and Proposition~\ref{lem:20160129a}\ref{lem:20160129ai}
(invoked for the metric $\scal{\cdot}{\cdot}_k$),
that $\lambda_k \geq \bar{\lambda}_k$. 
Moreover, by \eqref{eq:20160126b}, we have that if  $\bar{\lambda}_k < \bar{\lambda}$, 
then 
$L> 2 \delta \theta \nu_k/(\gamma_k \bar{\lambda}_k)$, hence
$\bar{\lambda}_k> 2 \delta \theta \nu_k/(L \gamma_k)$. 
Therefore, $\bar{\gamma}_k 
\geq \min\{\bar{\lambda}, 2\delta \theta \nu/(L \sup_{k \in \N} \gamma_k)\}$.\\
We now consider the case 
that $(x_k)_{k \in \N}$, $(\gamma_k)_{k \in \N}$, and ${(\lambda_k)}_{k \in \N}$
are defined according to \ref{LS3}.
Let $k \in \N$.
Then, for every $\gamma>0$,
$\norm{ \nabla f(J_k(x_k,\gamma,\lambda_k)) - \nabla f(x_k)} 
\leq L \norm{J_k(x_k,\gamma,\lambda_k) - x_k}$,
and, by \eqref{eq:norms} and Fact~\ref{p:20160125f},
\begin{equation}
\label{eq:20160121b}
\big\lVert \nabla^k f(J_k(x_k,\gamma)) - \nabla^k f(x_k)\big\rVert_k 
\leq \frac{L}{\nu_k} \norm{J_k(x_k,\gamma) - x_k}_k.
\end{equation}
Now, define
\begin{equation}
\label{eq:20160104c}
\bar{\gamma}_k = \max\big\{ \gamma \in \RPP \,\vert\, 
(\exists\, i \in \N)(\gamma = \bar{\gamma} \theta^i)\  L \leq \nu_k\delta/(\gamma \lambda_k) \big\}.
\end{equation}
It follows from \eqref{eq:20160121b} and 
the definition of $\gamma_k$ in \ref{LS3}
that $\gamma_k \geq \bar{\gamma}_k$.
Moreover, by \eqref{eq:20160104c}, 
if $\bar{\gamma}_k<\bar{\gamma}$, then 
$L> \delta \theta \nu_k/(\bar{\gamma}_k \lambda_k)$,
 hence $\bar{\gamma}_k> \delta \theta \nu_k/(L \lambda_k)$.
 Therefore $\bar{\gamma}_k \geq \min\{\bar{\gamma}, \delta \theta \nu/(L \sup_{k \in \N} \lambda_k)\}$.

\ref{p:20151228aii}:
Assume first that 
$(x_k)_{k \in \N}$, $(\gamma_k)_{k \in \N}$, and ${(\lambda_k)}_{k \in \N}$
are defined according to either \ref{LS1} or \ref{LS3}. Since $S_* \neq \varnothing$,
$\inf_\hh (f+g)>-\infty$. Besides,
we derive from Theorem~\ref{thm:convergence}\ref{thm:convergencei}
that $x_{k} \rightharpoonup \bar{x} \in S^*$.
Let, for every $k \in \N$, $\tilde{\gamma}_k = \gamma_k/\theta$ and 
$\tilde{x}_{k} = J(x_k,\tilde{\gamma}_k, \lambda_k)$. Then, 
by Lemma~\ref{lem:Jproperties}\ref{lem:Jpropertiesi} (applied to each metric $\scal{\cdot}{\cdot}_k$), for every $k \in \N$
\begin{equation}
\label{eq:20160126d}
\norm{J_k(x_k,\tilde{\gamma}_k, \lambda_k) - x_k}_k 
\leq \frac{\tilde{\gamma}_k}{\gamma_k} \norm{J_k(x_k, \gamma_k, \lambda_k) - x_k}_k 
= \frac{1}{\theta}\norm{x_{k+1} - x_k}_k.
\end{equation}
Moreover, according to the definition of $\gamma_k$
 in \ref{LS1} and \ref{LS3}, we have 
 respectively
\begin{equation}
\label{eq:20151230eb}
(\forall\, k \in \N)\ \gamma_k < \bar{\gamma} \Rightarrow\  
f(\tilde{x}_{k}) - f(x_k) - \big\langle \tilde{x}_{k} -x_k\, |\, \nabla^k f(x_k)\big\rangle_k 
> \delta\frac{\norm{\tilde{x}_{k} - x_k}_k^2}{\tilde{\gamma}_k \lambda_k}
\end{equation}
or
 \begin{equation}
\label{eq:20151230ea}
(\forall\, k \in \N)\ \gamma_k < \bar{\gamma} \Rightarrow\  
\big\lVert \nabla^k f(\tilde{x}_{k}) - \nabla^k f(x_k)\big\rVert_k 
> \delta \frac{\norm{\tilde{x}_{k} - x_k}_k}{\tilde{\gamma}_k \lambda_k}.
\end{equation}
Now, by Theorem~\ref{thm:convergence}\ref{thm:convergenceii01}
and \eqref{eq:norms}, 
$\norm{x_{k+1} - x_{k}}_k \to 0$,
hence, by \eqref{eq:20160126d} and \eqref{eq:norms}, $\tilde{x}_{k} - x_k \to 0$. Then,
since $x_{k} \rightharpoonup \bar{x}$,
Lemma~\ref{l:20160103a}\ref{l:20160103aii} yields that  
$\exists\, L>0$ 
such that, $\forall\, k \in \N$,
$\norm{ \nabla f(\tilde{x}_k) - \nabla f(x_k) }  
\leq L  \norm{\tilde{x}_k - x_k}$
and
$f(\tilde{x}_k) - f(x_k) - \scal{\tilde{x}_k - x_k}{\nabla f(x_k)}  \leq L
\norm{\tilde{x}_k - x_k}^2/2$.
The above inequalities, in view of \eqref{eq:norms}, \eqref{eq:20160204b}, 
and Fact~\ref{p:20160125f},
imply
\begin{align}
\label{eq:20151230za}
(\forall\, k \in \N)\quad &\norm{ \nabla f(\tilde{x}_k) - \nabla f(x_k) }  
\leq \frac{L}{\nu_k}  \norm{\tilde{x}_k - x_k}\\
\label{eq:20151230z}
(\forall\, k \in \N)\quad 
&f(\tilde{x}_k) - f(x_k) - \big\langle \tilde{x}_k - x_k\,|\,\nabla^k f(x_k)\big\rangle_k  \leq \frac{L}{2 \nu_k} 
\norm{\tilde{x}_k - x_k}_k^2.
\end{align}
Thus, \eqref{eq:20151230z} and \eqref{eq:20151230eb} yield
that, $\forall\, k \in \N$, $\gamma_k < \bar{\gamma}\ \Rightarrow\ $
$L/(2\nu_k) \geq \delta /(\tilde{\gamma}_k \lambda_k)\ \Rightarrow\ $ 
$\gamma_k = \tilde{\gamma}_k \theta \geq 2\delta \theta \nu_k/(L \lambda_k)$.
Thus, $\forall\, k \in \N$, $\gamma_k 
\geq \min\{\bar{\gamma}, 2\delta \theta \nu/(L \sup_{k \in \N} \lambda_k)\}>0$.
Moreover, it follows from \eqref{eq:20151230za} and \eqref{eq:20151230ea}
that, $\forall\, k \in \N$, $\forall\,\gamma_k<\sigma\ \Rightarrow\ $
$L/\nu_k \geq \delta /(\tilde{\gamma}_k \lambda_k)\ \Rightarrow\ $ 
$\gamma_k = \tilde{\gamma}_k \theta \nu_k \geq \delta \theta/(L \lambda_k)$. 
Thus, $\forall\,k \in \N$, $\gamma_k 
\geq \min\{\bar{\gamma}, \delta \theta \nu/(L \sup_{k \in \N} \lambda_k)\}>0$.\\
The case \ref{LS2} or \ref{LS4} is treated in the same way.
\end{proof}

\begin{remark}
\normalfont
\label{rmk:20160110a}\
\begin{enumerate}[(i)]
\item\label{rmk:20160110aii}
The results given in Proposition~\ref{p:20151228a}, shows that
when $\nabla f$ has some kind of Lipschitz continuity property,
the stepsizes determined by \ref{LS3} may be half of those determined
by the other line search methods. In particular, 
if $\nabla f$ is  $L$-Lipschitz continuous on $\dom g$ 
 and $(\gamma_k)_{k \in \N}$
are defined according to \ref{LS3},
we can make $\inf_{k \in \N} \gamma_k$ 
(by choosing $\lambda_k \equiv 1$, and $\delta$ and $\theta$ sufficiently 
close to $1$) arbitrarily close to
$\nu/L$; whereas using \ref{LS1}, \ref{LS2}, or \ref{LS4},
 $\inf_{k \in \N} \gamma_k$ can approach
 $2 \nu/L$. This latter result is in line with the state of the art 
convergence theory on 
forward-backward splitting algorithm (with $\nu = 1$) \cite{Bre09,Smms05,Dav2015}.
This suggests that \ref{LS3} is not quite appropriate
for  (VM-FBS).
\item Similar results recently appeared under stronger hypotheses
and for more specific cases
 in \cite{Cruz15} 
and \cite{Bon15}.
In particular in \cite[Proposition~4.4]{Cruz15} (for stationary metrics) it is proved that 
if $f$ is globally Lipschitz continuous and \ref{LS3} is used, 
$\inf_{k \in \N} \gamma_k$ can (only) reach $1/(2L)$. Moreover,
it is also showed that
 $\inf_{k \in \N} \gamma_k>0$,
if $\hh$ is finite-dimensional and $\nabla f$ is locally Lipschitz continuous
around any point of $S_*$. However, if $\hh$ is finite dimensional and
 $\nabla f$ is locally Lipschitz continuous around any point of $\dom g$, 
then $\nabla f$ is Lipschitz continuous on any (weakly) compact
subset of $\dom g$ (recall also 
Remark~\ref{rmk:20151230a}\ref{rmk:20151230ai}). 
Therefore Proposition~\ref{p:20151228a}\ref{p:20151228aii}
encompasses \cite[Proposition~4.4(ii)]{Cruz15}.
On the other hand in \cite{Bon15}
a variable metric projected gradient method is studied
for finite dimensional convex problems
using the corresponding specialization of \ref{LS4}.
In Proposition~2.2 they prove that the $\lambda_k$'s are bounded away from zero
provided that
$\nabla f$ is locally Lipschitz continuous 
and additionally that $f$ is coercive.
\end{enumerate}
\end{remark}

\section{Dealing with general placements of the domains}
\label{sec:gendomains}

This section provides a slight variation of algorithm (VM-FBS)
that can handle more general configurations of the domains of $f$ and $g$. 
This will be done at the expense of additional assumptions. We recall that in {\bf H2} it is required 
that $\dom g \subset \dom f$; however that condition is not always appropriate when
$f$ is taken of  divergence type \cite{Bon12,Sal14}.
Here we replace 
 assumptions {\bf H1} and {\bf H2}  by the following {\bf H1$^\prime$}--{\bf H3$^\prime$}. 
For every $h \colon \hh \to \RX$ and  $\alpha \in \R$,
 we set $\{h \leq \alpha\}=\{ x \in \hh\,\vert\, h(x) \leq \alpha\}$.
 Moreover, we define the distance between subsets $A$ and $B$ of $\hh$
as $d(A,B) = \inf\big\{\norm{x - y}~\vert~x \in A\ \text{and}\ y \in B\big\}$.

\begin{description}[itemsep=0mm]
\item[H1$^\prime$] $f\colon \hh \to \RX$ and $g\colon \hh \to \RX$ 
are proper convex and lower semicontinuous functions, bounded from below and such that
$\dom g \cap \mathrm{int}\dom f \neq \varnothing$;
\item[H2$^\prime$] $f$ is G\^ateaux differentiable on 
$\dom g \cap \mathrm{int}\dom f$, $\nabla f$
is uniformly continuous on any weakly compact subset of 
$\dom g \cap \mathrm{int}\dom f$, and $\nabla f$ is bounded on any
sublevel sets of $f+g$.
\item[H3$^\prime$] 
for every $x \in \mathrm{int}\dom f \cap \dom g$, 
$\{f+g \leq (f+g)(x)\} \subset \mathrm{int}\dom f \cap \dom g$ and
 $d(\{f+g \leq (f+g)(x)\},\hh\setminus \mathrm{int}\dom f)>0$.
\end{description}

The following result clarifies the role of the above hypotheses.

\begin{proposition}
\label{prop:20160428a}
Suppose that {\bf H1$^\prime$}--{\bf H3$^\prime$}  hold.
Let $x_0 \in
 \mathrm{int}\dom f \cap \dom g$
and  $K_0 = \{f+g \leq (f+g)(x_0)\}$. Then there exist constants
$C_1$ and $C_2$ in $\RP$, such that, 
\begin{equation}
\label{eq:20160427b}
(\forall\, x \in K_0)(\forall\, \gamma \in \RPP)\quad 
\norm{J(x, \gamma,1) - x} 
\leq C_1 \gamma  + C_2 \sqrt{\gamma}.
\end{equation}
Moreover, setting 
$\delta_0 = d(K_0, \hh\setminus\mathrm{int}\dom f)>0$, we have
\begin{equation*}
(\forall\, x \in K_0)(\forall\, \gamma>0)\quad
J(x,\gamma,1) \notin \dom f\ \Rightarrow\ \gamma \geq c_0:= c(\delta_0)>0,
\end{equation*}
where $c(\cdot)$ is the inverse of the strictly incresing function 
$\gamma \mapsto  C_1 \gamma  + C_2\sqrt{ \gamma}$.
\end{proposition}
\begin{proof}
We first prove that $g$ is bounded on $K_0$.
Indeed, let $\vartheta \in \R$ be such $\vartheta \leq f$.
 Then, $\forall\, x \in K_0$, $\vartheta + g(x) \leq (f+g)(x) \leq (f+g)(x_0)$,
 hence $g(x) \leq (f+g)(x_0) - \vartheta$.
Concerning the first part of the statement,
we note that, since $g$ is bounded from below, we have $g^*(0)<+\infty$,
hence in \eqref{eq:20160213c} of Lemma~\ref{lem:Jproperties}\ref{lem:Jpropertiesiib}, 
we can take $u=0$, obtaining $\norm{J(x,\gamma,1) - x} \leq \gamma \norm{\nabla f(x)} 
+ (2 \gamma)^{1/2}\sqrt{g(x) + g^*(0)}$. Since,
by  {\bf H2}$^\prime$,
 $\nabla f$ is bounded on $K_0$ and $g$ is bounded on $K_0$ too, 
 \eqref{eq:20160427b} follows.
Moreover, if $x \in K_0$ and $J(x,\gamma,1) \notin \dom f$, then 
$\delta_0 \leq \norm{J(x,\gamma,1)- x} \leq C_1 \gamma + C_2 \sqrt{\gamma}$
and hence $c(\delta_0) \leq \gamma$.
\end{proof}

\begin{proposition}
\label{prop:20160213a}
Let $\mathcal{G}$ be a real Hilbert space, let 
$A\colon \hh \to \mathcal{G}$ be a non-zero bounded linear operator,
let $h \in \Gamma_0(\mathcal{G})$ and set $f = h \circ A$.
Let $g \in \Gamma_0(\hh)$ and suppose that
$\dom g \cap \mathrm{int}\dom f \neq \varnothing$ and 
$g$ is bounded from below.
Then, the following hold.
\begin{enumerate}[$(i)$]
\item\label{prop:20160213ai}  
$f$ is bounded from above on the sublevel sets of $f+g$.
Moreover, if $f$ is bounded from below, then $g$ is bounded from above
on the sublevel sets of $f+g$.
\item\label{prop:20160213aii} 
Suppose that $\dom f \neq \hh$, and that, 
\begin{enumerate}[$(a)$]
\item\label{prop:20160213a_a}
for every $\alpha \in \R$, $\{h \leq \alpha\} \subset \mathrm{int}\dom h$ and 
$d(\{ h \leq \alpha\}, \mathcal{G}\setminus\mathrm{int}\dom h)>0$.
\end{enumerate}
Then  {\bf H3$^\prime$}  holds.
\item\label{prop:20160213aiii} 
Suppose that $\mathcal{G}$ is finite dimensional, $h$ is coercive,
 and $\forall\, \alpha \in \R$, $\{h \leq \alpha\} \subset \mathrm{int}\dom h$.
 Then $f$ is bounded from below and \ref{prop:20160213a_a} is satisfied;
 hence {\bf H1$^\prime$} and {\bf H3$^\prime$} hold. 
\item\label{prop:20160213aiibis} 
In addition to the assumptions in \ref{prop:20160213aiii}, suppose 
that $\dom h = \mathrm{int}\dom h$ and that $\nabla h$
is continuously differentiable on $\dom h$.
Then  {\bf H2$^\prime$} holds.
\end{enumerate}
\end{proposition}
\begin{proof}
\ref{prop:20160213ai}:
Let $\alpha \in \R$ and let
 $\beta \in \R$ be such that 
and $\beta \leq g$.
Then, for every $x \in \{f+g \leq \alpha\}$, $f(x) + \beta\leq (f + g)(x) \leq \alpha$,
 hence $f(x) \leq \alpha- \beta$.
 As regards the second part, let $\vartheta \in \R$ be such $\vartheta \leq f$.
 Then, for every $x \in \{f+g \leq \alpha\}$, $\vartheta + g(x) \leq (f+g)(x) \leq \alpha$,
 hence $g(x) \leq \alpha - \vartheta$.


\ref{prop:20160213aii}:
Let $x \in \mathrm{int}\dom f \cap \dom g$
and set $K= \{ f+g \leq (f+g)(x)\}$ and $U = \mathrm{int}\dom h$. 
It follows from \ref{prop:20160213ai} that $\exists\,\alpha \in \R$
such that, $f \leq \alpha$ on $K$.
Then  $K \subset \{ f \leq \alpha \} \subset A^{-1}(\{h \leq \alpha\}) \subset A^{-1} (U)$, 
which is open (and contained in $\dom f$),
hence $\{ f \leq \alpha \} \subset A^{-1} (U) \subset \mathrm{int}\dom f$.
Let $\delta = d(\{h \leq \alpha \}, \mathcal{G}\setminus U)>0$.
Now, let $x^\prime \in \{f \leq \alpha\}$ and $x^{\prime\prime} 
\in \hh \setminus \mathrm{int}\dom f$.
Then $Ax^\prime \in \{h \leq \alpha\}$ and $A x^{\prime\prime} \notin U$,
hence $\delta \leq \norm{A x^\prime- A x^{\prime\prime}} 
\leq \norm{A} \norm{x^\prime - x^{\prime\prime}}$.
Thus, $d(K,\hh\setminus\mathrm{int}\dom f)  \geq d(\{f \leq \alpha\}, \hh\setminus\mathrm{int}\dom f)\geq \delta/\norm{A}>0$.

\ref{prop:20160213aiii}:
Since  the range of $A$, $R(A)$, is closed in $\mathcal{G}$, $h + \iota_{R(A)}$ is lower semicontinuous and coercive,
hence it has a minimizer, say $\bar{y} \in R(A)$. Then, taking $\bar{x} \in \hh$ such that
$A \bar{x} = \bar{y}$, we have, for every $x \in \hh$, $h(A \bar{x}) \leq h(A x) = f(x)$.
Thus, $f$ is bounded from below.
Let $\alpha \in \R$. Since $\{h \leq \alpha\}$ is compact and 
$d(\cdot, \mathcal{G}\setminus \mathrm{int}\dom h)$ is continuous
and strictly positive on $\{h \leq \alpha\}$, 
then $d(\{h \leq \alpha\}, \mathcal{G}\setminus \mathrm{int}\dom h)>0$.

\ref{prop:20160213aiibis}:
Clearly $\dom f = A^{-1}(\dom h) = \mathrm{int}\dom f$ and $f$ is continuously
differentiable on $\dom f$ and, for every $x \in \dom f$, $\nabla f (x) = A^* \nabla h(A x)$.
Let $\alpha \in \R$. 
It follows from \ref{prop:20160213ai} that there exists $\beta \in \R$ such that
$f \leq \beta$ on $\{ f+g \leq \alpha\}$.
Then, for every $x \in \{ f+g \leq \alpha\}$,
$h(A x) \leq \beta$, therefore $A x \in \{h \leq \beta\}$. 
Since $\nabla h$ is continuous on the compact $\{h \leq \beta\}$, 
there exists $\eta \in \RP$ such that
$\norm{\nabla h} \leq \eta$ on $\{h \leq \beta\}$. Thus 
$\norm{\nabla f (x)} \leq \norm{A^*} \norm{\nabla h(A x)} \leq \norm{A^*} \eta$.
This proves that $\nabla f$ is bounded on the sublevel sets of $f+g$.
Let $K$ be a weakly compact subset of 
$\mathrm{int}\dom f = A^{-1} (\mathrm{int}\dom h)$. Since $A$ is weak-to-weak
continuous, $A(K)$ is a (weakly) compact subset of $\mathrm{int}\dom h$.
Thus, by the Heine-Cantor theorem, $\nabla h$ is uniformly continuous on $A(K)$.
Hence $\nabla f = A^* \circ \nabla h \circ A_{\lvert \dom f}$ is uniformly continuous on $K$.
\end{proof}

\begin{example}
\label{ex:bregman}
\normalfont
Let $n \in \N$, $n \geq 1$.
Let $\varphi$ be a Legendre function on $\R^n$ \cite{Bau06}
which is twice continuously differentiable on $\mathrm{int} \dom \varphi$, 
and let
\begin{equation}
\label{eq:bregman}
D\colon \R^n \times \R^n \to [0,+\infty]
\colon (x,y) \mapsto 
\begin{cases}
\varphi(x) - \varphi(y) - \scal{\nabla \varphi(y)}{x-y}& \text{if}\ y \in \mathrm{int}\dom \varphi\\
+\infty&\text{otherwise}
\end{cases}
\end{equation}
be the associated Bregman distance. 
Suppose that, for every $x \in \mathrm{int}\dom \varphi$, $D(x,\cdot) \in \Gamma_0(\R^n)$
and $D(x,\cdot)$ is coercive
--- this case is studied in \cite{Bau06} (see in particular Lemma~2.6) and 
occurs, e.g., for the Kullback-Leibler divergence, where 
$\varphi(x) = \sum_{i=1}^n x_i \log x_i - x_i$.
Let $b \in \mathrm{int}\dom \varphi$, let $A\colon \hh \to \R^n$ be 
a bounded linear operator, and set $f = D(b,A\cdot)$. 
Thus, in virtue of 
Proposition~\ref{prop:20160213a}\ref{prop:20160213aiii}-\ref{prop:20160213aiibis},
if $g \in \Gamma_0(\hh)$ is bounded from below and such that $A^{-1}(\mathrm{int}\dom \varphi) \cap \dom g \neq \varnothing$,
assumptions {\bf H1$^\prime$}--{\bf H3$^\prime$} are satisfied.

\end{example}

Under assumptions {\bf H1$^\prime$}--{\bf H3$^\prime$},
we can modify algorithm (VB-FBS), by adding a further
line search for computing $y_k$.
More precisely, for \ref{LS2} and \ref{LS4}, the sequence 
$(\gamma_k)_{k \in \N}$ cannot be chosen a priory anymore,
but it has to be computed by the following procedure.
Let $\overline{\gamma}>0$, let $x_0 \in \mathrm{int}\dom f \cap \dom g$,
and set $K_0 = \{f+g \leq (f+g)(x_0)\}$.
Then, for every $k \in \N$, assume that $x_k \in K_0$, and compute
\begin{equation}
\label{eq:20160213h}
\gamma_k 
= \max \Big\{ \gamma \in \RPP~\big\vert~(\exists i \in \N)(\gamma = \bar{\gamma}\theta^i) 
\quad J(x_k, \gamma,1) \in \dom f\Big\}.
\end{equation}
Note that, since $x_k \in \mathrm{int}\dom f$ and $J(x_k,\gamma,1)\to x_k$ 
as $\gamma \to 0$, the procedure \eqref{eq:20160213h}
is well-defined.
Moreover, because of 
Proposition~\ref{prop:20160428a},
if $\gamma_k \leq \bar{\gamma}\theta$, then 
$J(x_k, \gamma_k/\theta,1) \notin \dom f$ and hence $\gamma_k\geq c_0\theta$.
Therefore $\inf_{k} \gamma_k>0$.
Procedure \eqref{eq:20160213h} ensures that
$y_k \in \dom f \cap \dom g$ and the subsequent line search makes sense. Moreover,  Proposition~\ref{lem:20160129a} yields that 
$(f+g)(x_{k+1}) \leq (f+g)(x_k)$, hence $x_{k+1} \in K_0$ and the algorithm can continue.
Concerning \ref{LS1} and \ref{LS3},
the $\gamma_k$ computed by \eqref{eq:20160213h}
is meant to replace $\bar{\gamma}$ in \ref{LS1} and \ref{LS3},
meaning that, for every $k \in \N$, 
they will do backtracking on  $\gamma$ starting from
the output of \eqref{eq:20160213h}.
This will make sense of the subsequent procedures \ref{LS1} and \ref{LS3}.
Again Proposition~\ref{lem:20160129a} proves that the next step is  descendent
and hence $x_{k+1} \in K_0$.
Note that in \ref{LS1}, if ${(\lambda_k)}_{k \in \N} \equiv 1$, then one can
perform \ref{LS1} only, since it will automatically search for a point in $\dom f$.

\section{Applications}
\label{sec:app}

In this section we  illustrate several models that
can be tackled by the proposed algorithm.
In particular we show that its scope of applicability encompasses 
problems that involve
Banach spaces or functions of divergence type.
To that purpose we recall few facts.
\begin{fact}
\label{prop:chainrule}
Let  $A\colon \hh \to \mathcal{B}$ be a bounded linear operator
between a real Hilbert space and a real Banach space.
Let $\varphi\colon \mathcal{B} \to \R$ be a differentiable 
function and suppose that its derivative $\varphi^\prime$ is 
$\alpha$-H\"older continuous on bounded sets,
for some $\alpha \in \left]0,1\right]$,
that is, for every bounded set $Y \subset \B$,
there exists $C \in \RP$ such that
\begin{equation*}
(\forall\,(y_1, y_2) \in Y^2)\qquad 
\norm{\varphi^\prime (y_1) - \varphi^\prime(y_2)}_{\B^*} \leq C \norm{y_1 - y_2}_{\B}^\alpha.
\end{equation*}
Then, $f = \varphi \circ A \colon\hh \to \R$ is 
differentiable and
$ \nabla f = A^* \circ \varphi^\prime \circ A$, 
where $A^* \colon \B^* \to \hh$ is the adjoint of $A$.
Moreover, $\nabla f$ is $\alpha$-H\"older continuous on the bounded sets of $\hh$.
\end{fact}

The following result is in \cite[Corollary~2.44 and Theorem~2.53(f)]{Sch2012}

\begin{fact}
\label{f:JpHolder}
Let $\B$ be a uniformly smooth Banach space, let 
$p \in \left]1,+\infty\right[$ and set $\varphi = (1/p) \norm{\cdot}^p_{B}$.
Then $\varphi$ is  
differentiable and $ \varphi^\prime = J_{\B,p}$ is the $p$-duality mapping of $\B$, 
which is uniformly continuous on bounded sets.
Moreover, if $\B$ has modulus of smoothness of power type $q \in \left]1,2\right]$,
then, $\varphi^\prime$ is $(p-1)$-H\"older continuous, if $p \leq q$,
and $(q-1)$-H\"older continuous on bounded sets, if $p>q$.
\end{fact}

In the following we give a prominent example in which
the duality map of the involved Banach space is explicitly computable.

\begin{remark} 
\label{rmk:20160502a}
\normalfont
Let $(\Omega, \mathfrak{A}, \mu)$ be a $\sigma$-finite measure space
and let $p \in \left]1,+\infty\right[$.
Then 
$L^p(\Omega, \mu)$ is uniformly smooth with modulus of smoothness
of power type $\min\{2,p\}$ \cite{LindTzaf1979}. Therefore,
it follows from Fact~\ref{f:JpHolder} that
the function $\varphi = (1/p) \norm{\cdot}^p_p$
is differentiable and $\varphi^\prime$
is $(p-1)$-H\"older continuous on $L^p(\Omega, \mu)$, if $p \leq 2$,
and Lipschitz continuous on the bounded sets of $L^p(\Omega, \mu)$, if $p>2$.
Moreover, for every $x \in L^p(\Omega, \mu)$, $\varphi^\prime(x) \in L^{p^*}(\Omega,\mu)$
and
$
\varphi^\prime(x)\colon \Omega \to \R\colon \omega \mapsto \abs{x(\omega)}^{p-1} 
\mathrm{sign}(x(\omega))
$.
\end{remark}

It follows from Fact~\ref{f:JpHolder} and Fact~\ref{prop:chainrule} that the following general optimization
problem is of the form (P) and hypotheses {\bf H1} and {\bf H2} are satisfied.
\smallskip
\begin{problem}
\label{prob:20160203a}
Let $A\colon \hh \to \B$
be a bounded linear operator between a real Hilbert space and a 
real uniformly smooth Banach space. Let
 $g \in \Gamma_0(\hh)$ and $b \in \B$.
Then
\begin{equation*}
\minimize{x \in \hh}{\frac{1}{p}\norm{A x - b}_\B^p + g(x)}
\qquad(p>1).
\end{equation*}
\end{problem}

\begin{remark}
\normalfont
In Problem~\ref{prob:20160203a},
 we have $\nabla \big((1/p)\norm{A \cdot - b}_\B^p\big)(x) = A^*J_{\B,p}(A x - b)$.
In this case the gradient descent step in (VM-FBS)  
resembles the Landweber step in Banach spaces \cite{Bre09,Sch2012}.
\end{remark}

Based on Remark~\ref{rmk:20160502a}, we give some
significant instances of Problem~\ref{prob:20160203a}.

\begin{example}
\normalfont
Let $\hh$ be a real Hilbert space and let ${(e_k)}_{k \in \KK} \in \hh^{\KK}$
be an orthonormal basis of $\hh$.
Let $(\Omega, \mathfrak{A}, \mu)$ be a $\sigma$-finite measure space
and let $p \in \left]1,+\infty\right[$. Let $A \colon \hh \to L^p(\Omega,\mu)$
be a bounded linear operator, let $b \in L^p(\Omega,\mu)$, and let $(g_k)_{k \in \N}$
be a sequence of functions in $\Gamma_0(\R)$ such that
$g_k \geq g_k(0) = 0$, for every $k \in \N$.
Then
\begin{equation}
\label{eq:20160302f}
\minimize{x \in \hh}{\frac{1}{p}\norm{A x - b}_p^p + \sum_{k \in \N} g_k (\scal{e_k}{x})}.
\end{equation}
Denoting by $f$ and $g$ respectively the first and second term 
in \eqref{eq:20160302f}, we have $\nabla f(x) = A^* u$, where
$
u\colon\Omega \to \R\colon \omega \mapsto 
\abs{(Ax)(\omega) - y(\omega)}^{p-1} \mathrm{sign}\big((Ax)(\omega) - b(\omega)\big)
$,
and the proximity operator can be computed component-wise \cite{Smms05}, that is
\begin{equation*}
\prox_{\gamma g}(x) = {\big( \prox_{\gamma g_k} (\scal{e_k}{x}) \big)}_{k \in \N}.
\end{equation*}
Moreover, it follows from Remark~\ref{rmk:20160502a} that $\nabla f$
is $(p-1)$-H\"older continuous if $p\leq 2$ and Lipschitz continuous on bounded sets if $p>2$.
Therefore, it follows from Theorem~\ref{thm:convergence}
and Proposition~\ref{p:20151228a}\ref{p:20151228aii}
that, in this case, the sequence generated by (VM-FBS) is weakly convergent
to a solution of \eqref{eq:20160302f} and converges in function values to the corresponding minimum. Moreover, if $p \geq 2$,  the convergence in function values boasts a rate of $o(1/k)$.
This example covers the class of problems approached by the iterative 
shrinkage/thresholding algorithm 
\cite{Dau04,Tebu09},
but here a more general discrepancy term is used.
A special case of \eqref{eq:20160302f} is
\begin{equation}
\label{eq:20160204a}
\min_{x \in \ell^2(\KK)} \frac{1}{p}\norm{A x - b}_p^p + \norm{x}_1,
\end{equation}
where, $\KK$ is a countable set, 
$A \colon \ell^2(\KK) \to \R^n$  is a bounded linear operator, and $b \in \R^n$.
Note that here $\dom \norm{\cdot}_1 = \ell^1(\KK)$, which is not closed in $\ell^2(\KK)$.
We highlight that  problems of type \eqref{eq:20160204a}
arise in function interpolation (from discrete data) and non parametric function estimation
(support vector regression).
\end{example}

We end the section by showing a prototype of problems where $\dom g \not\subset \dom f$
and hypotheses {\bf H1$^\prime$}--{\bf H3$^\prime$} in Section~\ref{sec:gendomains}
are met (recall Example~\ref{ex:bregman}).

\begin{problem}
\label{prob:20160407a}
Let $\hh$ be a real Hilbert space, let $n \in \N$ with $n \geq 1$, and let $\varphi$ be a Legendre function
on $\R^n$
such that $\varphi$ is twice continuously differentiable on $\mathrm{int}\dom \varphi$ and
its associated Bregman distance $D$ $(\text{see \eqref{eq:bregman}})$ satisfies the 
condition: $\forall\, z \in \mathrm{int}\dom \varphi$, $D(z,\cdot) \in \Gamma_0(\R^n)$ 
and $D(z,\cdot)$ is coercive. Let  $A\colon \hh \to \R^n$
be a bounded linear operator, let $b \in \mathrm{int}\dom \varphi$, and let
 $g \in \Gamma_0(\hh)$ be such that 
 $A^{-1}(\mathrm{int}\dom \varphi) \cap \dom g \neq \varnothing$
 and $g$ is bounded from below.
Then
\begin{equation*}
\minimize{x \in \hh}{D(b,Ax) + g(x)}.
\end{equation*}
\end{problem}

Examples of Problem~\ref{prob:20160407a} are provided in the following.

\begin{example}
\normalfont
Let $n \in \N$ with $n \geq 1$ and let 
$D(z,y) = \sum_{i=1}^n z_i \log(z_i/y_i) + y_i - z_i$ be
the Kullback-Leibler divergence. Then, let $A \in \R_+^{n \times n}$, $b \in \RPP^n$ and solve
\begin{equation}
\label{eq:20160411a}
\min_{x \in \R^n_+} D(b,Ax) + TV(x).
\end{equation}
where $TV$ is the (discrete) total variation.
Note that in this case $g = \iota_{\R^n_+} + TV$.
Moreover $\dom g \not\subset \dom D(b,\cdot)\circ A$
and  $\nabla (D(b,\cdot)\circ A)$ is only locally Lipschitz continuous on its domain. 
Then, it follows from Theorem~\ref{thm:convergence}
and Proposition~\ref{p:20151228a}\ref{p:20151228aii} that (VM-FBS)
(with the additional line search presented in Section~\ref{sec:gendomains}) provides a sequence which converges to a solution of \eqref{eq:20160411a}
and converges in functional values to the related minimum at rate $o(1/k)$.
Problem \eqref{eq:20160411a} 
is of the type considered in \cite[Section~4.2]{Bon12},
but here the introduction of the background signal is avoided
--- provided that the sought signal $x^*$ satisfies $A x^*>0$.
Another instance of Problem \ref{prob:20160407a}
is
\begin{equation*}
\min_{\substack{x \in \R^n_+ \\ \beta \in \R_+}} D(b,Ax + \beta \mathbf{1}) + \norm{x}_1,
\end{equation*} 
where the signal and the background are sought
and $\mathbf{1}$ is the vector of $\R^n$ of all ones.
Here again the domain of the map $(x,\beta) \mapsto D(b,Ax + \beta \mathbf{1})$ is not contained
in $\RP^n \times \RP$.

\end{example}


\end{document}